\documentclass[draft, reqno]{amsart}
\usepackage[all]{xy}                        %

\CompileMatrices                            

\UseTips                                    

\input xypic
\usepackage[bookmarks=true]{hyperref}       

\usepackage{amssymb,latexsym,amsmath,amscd}
\usepackage{xspace}
\usepackage{enumerate}
\usepackage{graphicx}
\usepackage{vmargin}
\usepackage{todonotes}

\DeclareMathOperator{\red}{red}
\DeclareMathOperator{\Res}{Res}
\DeclareMathOperator{\reg}{reg}

\reversemarginpar

\theoremstyle{plain}
\newtheorem{theorem}{Theorem}[section]
\newtheorem*{theorem*}{Theorem}
\newtheorem{proposition}[theorem]{Proposition}
\newtheorem{corollary}[theorem]{Corollary}
\newtheorem{lemma}[theorem]{Lemma}

\theoremstyle{definition}

\newtheorem{notation}[theorem]{Notation}

\newtheorem{remark}[theorem]{Remark}

\newcommand{\enm}[1]{\ensuremath{#1}}          %

\newcommand{\cal}[1]{\mathcal{#1}}

\newcommand{\NN}{\enm{\mathbb{N}}}

\newcommand{\ZZ}{\enm{\mathbb{Z}}}

\newcommand{\PP}{\enm{\mathbb{P}}}

\newcommand{\KK}{\enm{\mathbb{K}}}

\newcommand{\Ff}{\enm{\cal{F}}}

\newcommand{\Ii}{\enm{\cal{I}}}

\newcommand{\Ll}{\enm{\cal{L}}}

\newcommand{\Oo}{\enm{\cal{O}}}

\newcommand{\Rr}{\enm{\cal{R}}}
\newcommand{\Ss}{\enm{\cal{S}}}

\newcommand{\Zz}{\enm{\cal{Z}}}

\renewcommand{\phi}{\varphi}
\renewcommand{\theta}{\vartheta}
\renewcommand{\epsilon}{\varepsilon}

\begin{document}

\title[tangential variety]{On the secant varieties of tangential varieties}
\author{Edoardo Ballico}
\address{Dept. of Mathematics\\
 University of Trento\\
38123 Povo (TN), Italy}
\email{ballico@science.unitn.it}
\thanks{The author was partially supported by MIUR and GNSAGA of INdAM (Italy).}

\begin{abstract}
Let $X\subset \PP^r$ be an integral and non-degenerate variety. Let $\sigma _{a,b}(X)\subseteq \PP^r$, $(a,b)\in \NN^2$, be the join of $a$ copies of $X$ and $b$ copies of the
tangential variety of $X$. Using the classical Alexander-Hirschowitz theorem (case $b=0$) and a recent paper by H. Abo and N. Vannieuwenhoven (case
$a=0$) we compute $\dim \sigma _{a,b}(X)$ in many cases when $X$ is the $d$-Veronese embedding of $\PP^n$. This is related to
certain additive decompositions of homogeneous polynomials. We give a general theorem proving that $\dim \sigma _{0,b}(X)$ is
the expected one when $X=Y\times \PP^1$ has a suitable Segre-Veronese style embedding in $\PP^r$. As a corollary we prove that
if $d_i\ge 3$, $1\le i \le n$, and $(d_1+1)(d_2+1)\ge 38$ the tangential variety of $(\PP^1)^n$ embedded by $|\Oo _{(\PP
^1)^n}(d_1,\dots ,d_n)|$ is not defective and a similar statement for $\PP^n\times \PP^1$. For an arbitrary $X$ and an ample
line bundle $L$ on $X$ we prove the existence of an integer $k_0$ such that for all $t\ge k_0$ the tangential variety of $X$ with
respect to $|L^{\otimes t}|$ is not defective.
\end{abstract}

\maketitle

\section{Introduction}

For any integral and non-degenerate $n$-dimensional projective variety $X\subset \PP^r$ defined over an algebraically closed
field
$\KK$ with characteristic $0$ and any integer $a>0$ the $a$-th secant variety $\sigma _a(X)$ of $X$ is the closure in $\PP^r$
of the union of all linear subspaces of $\PP^r$ spanned by $a$ points of $X$. Its expected dimension is $\min \{r,a(n+1)-1\}$
and by Terracini's lemma (\cite[Corollary 1.11]{a}) we may check if $\sigma _a(X)$ has the expected dimension using the integer
$h^0(\Ii _Z(1))$, where $Z\subset X$ is a certain zero-dimensional scheme. If each $\sigma _a(X)$ has the expected dimension
we say that $X$ is \emph{not defective}. The tangential variety
$\tau (X)\subseteq X$ is the closure in $\PP^r$ of the union of all $T_qX$, $q\in X_{\mathrm{reg}}$, where $T_qX\subseteq
\PP^r$ denote the Zariski tangent space of $X$ at $q$. The variety $\tau (X)$ has dimension at most $2n$ and equality
holds if and only if a general tangent space of $X$ is tangent only at finitely many points of $X$. A key observation made in
\cite{cgg} (and then used several times, e.g. in \cite{bcgi}) is that to give
$\dim
\sigma _b(\tau (X))$ it is sufficient to compute $h^0(\Ii _{Z_{0,b}}(1))$, where $Z_{0,b}$ is a certain zero-dimensional scheme with $\deg (Z_{0,b}) =b(2n+1)$. Any scheme $Z_{0,1}\subset X_{\reg}$ is called a \emph{tangential scheme}. Using
this idea H. Abo and N.  Vannieuwenhoven proved a conjectural result of \cite{bcgi}: the description of all integers $\dim
\sigma _b(\tau (X))$,
where $X\subset \PP^r$, $r =-1+\binom{n+d}{n}$, is the $d$-Veronese embedding of $\PP^n$. As explained in \cite{av, bcgi, cgg}
this result has an interpretation in terms of the existence of a certain additive decomposition of a general homogeneous
degree $d$ polynomial in
$n+1$ variables. For any $(a,b)\in \NN^2\setminus \{(0,0)\}$ let $\sigma _{a,b}(X)$ denote the join of $X$ copies of $a$ and $b$ copies of $\tau (X)$ (\cite{a}). The interpretation of $\dim \tau (X)$ as the integer $r+1-h^0(\Ii _{Z_{0,1}}(1))$, where $Z_{0,1}$ is a certain zero-dimensional subscheme of $X$ with $\deg (Z_{0,1})=2n+1$, and the Terracini's lemma for joins (\cite[Corollary 1.11]{a}) give
$\dim \sigma _{a,b}(X) = r-h^0(\Ii _{Z_{a,b}}(1))$, where $Z_{a,b}\subset X$ is a certain zero-dimensional scheme. Since $Z_{a,b}$ is zero-dimensional and $\deg (Z_{a,b}) =a(n+1)+(2n+1)b$, we have $h^0(\Ii _{Z_{a,b}}(1)) = r+1-a(n+1)-b(2n+1)+h^1(\Ii _{Z_{a,b}}(1))$. It is well-known that the integer  $\dim \sigma _{a,b}(X)$
has an interpretation in terms of another related additive decomposition of
homogeneous degree $d$ polynomials. In particular $\sigma _{a,b}(X) =\PP^r$ if and only if for a general degree
$d$ homogeneous polynomial $f$ there are $a+2b$ homogeneous linear forms $\ell _1,\dots ,\ell _{a+2b}$ such that
\begin{equation}\label{eqiii1} f = \sum _{i=1}^{a}\ell _i^d + \sum _{i=1}^{b} \ell _{i+a}^{d-1} \ell _{i+a+b}.
\end{equation}
Concerning these decompositions  we prove the following result.

\begin{theorem}\label{uu1}
Fix integers $n\ge 2$, $t\ge 3$. Let $X\subset \PP^r$, $r= \binom{n+t}{n}-1$, denote the order $t$ Veronese embedding of $\PP^n$. Fix $(a,b)\in \NN^2$ such that either $a\ge \lceil \binom{n+3}{3}/(n+1)\rceil$ or $b\ge
\lceil \binom{n+3}{3}/(2n+1)\rceil$. Then either
$h^1(\Ii _{Z_{a,b}}(t)) =0$ or $h^0(\Ii _{Z_{a,b}}(t)) =0$, except in the cases $$(n,t,a,b)
\in \{(2,4,5,0),(2,3,0,2),(3,3,3,0),(3,4,9,0),(3,4,7,1),(4,3,7,0),(4,4,9,0)\}.$$
\end{theorem}
For the quadruples $(n,a,b,t)$ for which Theorem \ref{uu1} may be applied, it gives that a general homogeneous polynomial $f$ has a
decomposition (\ref{eqiii1}) for some $\ell _1,\dots ,\ell _{a+2b}$ (case $h^0(\Ii _{Z_{a,b}}(t))=0$), or the dimension of the
set of the polynomials $f$ in (\ref{eqiii1}) for some $\ell _1,\dots ,\ell _{a+2b}$ (case $h^1(\Ii _{Z_{a,b}}(t))=0$).

For $n=2,3,4$ we give all integers $\dim \sigma _{a,b}(\PP^n)$ (Propositions \ref{ux2} and \ref{ux3}).

The exceptional cases with either $a=0$ or $b=0$ in the list of Theorem \ref{uu1} have a well-known geometric interpretation.
The only case in this list with $ab\ne 0$ has an obvious geometric interpretation. Indeed, in $\PP^3$ we have $|\Ii _{Z_{7,1}}(4)|
=\{2Q\}$, where $Q$ is the only quadric containing the $7$ points which are the reduction of the $7$ degree $4$ connected
components of $Z_{7,1}$ and the tangent vector used to define the degree
$7$ connected component of
$Z_{7,1}$ (Proposition \ref{ux3}).

For more about additive decompositions of homogeneous polynomials and their use in applied mathematics, see \cite{ik, l} and
the introduction of \cite{av}.

\begin{theorem}\label{i2}
Let $Y$ be an integral projective variety, $\Ll \in \mathrm{Pic}(X)$ and $V\subseteq H^0(\Ll)$ a linear subspace inducing a
rational map $Y \dasharrow \PP^{\alpha -1}$, $\alpha := \dim V$, with image $Y'$ of dimension $n-1 = \dim Y$. Assume that
the tangential variety of $Y'$ is not defective.  Let $C$ be an integral projective curve, $\Rr\in \mathrm{Pic}({C})$ and
$W\subseteq H^0(\Rr)$ a linear subspace with $\dim W \ge 4$. Set $n:= 1+\dim Y$
$\alpha:=
\dim V$ and assume $\alpha \ge 4n^2  +2n -4$. Set
$X:= Y\times C$. Then the tangential variety of
$X$ with respect to
$V\boxtimes W$ is not defective.
\end{theorem}
We explain with full details in Remark \ref{ff0} the statement of Theorem \ref{i2}. The reader may take $Y\subset \PP (V)$, $C\subset \PP (W)$ and then $X$ embedded in $\PP (V\otimes W)$ by the Segre embedding of $\PP (V)\times \PP (W)$ in $\PP (V\otimes W)$.

As an immediate corollary of \cite{co} and Theorem \ref{i2} obtained taking $Y=(\PP^1)^{n-1}$,
$V:= H^0(\Oo_{(\PP^1)^{n-1}}(d_1,\dots ,d_{n-1}))$,
$C =\PP^1$ and $W:= H^0(\Oo_{\PP^1}(d_n))$ we get the following result.

\begin{corollary}\label{i2.0}
Fix integers $n\ge 2$ and $d_i>0$, $1\le i \le n$, such that $d_i\ge 3$ for all $i\ge 3$ and $(d_1+1)(d_2+1)\ge 38$. Then
the tangential variety of $(\PP^1)^n$ with respect to $H^0(\Oo _{(\PP^1)^n}(d_1,\dots ,d_n))$ is not defective.
\end{corollary}

See \cite{ab, abre, lp} for the dimensions of the secant varieties of Segre-Veronese varieties.

Theorem \ref{i2} has also the following immediate consequence for $\PP^m\times \PP^1$.

\begin{corollary}\label{i2.1}
Fix integers $m\ge 2$, $d\ge 3$ and $t\ge 3$ such that $(t+1)\binom{m+d}{m} \ge 4m^2+10m+2$; if $m\le 4$ assume $d\ne 3$. Then the
tangential variety of the embedding of
$\PP^m\times
\PP^1$ by the complete linear system $|\Oo _{\PP^m\times \PP^1}(d,t)|$ is not defective.
\end{corollary}

For the cases $m \le 4$, $t=3$ and $d=1,2,3$, see Propositions \ref{x1} and \ref{c4}.

\begin{remark}\label{nni2.3}
Take $X:= \PP^m\times (\PP^1)^{n-m}$, $n>m>0$, embedded by the complete linear system $|\Oo _X(d_1,\dots, d_{n-m+1})|$ with
$d_i\ge 3$ for all $i$ and $\binom{m+d_1}{m} \ge 4m^2+10m+2$. By induction on $n-m$ we may apply Theorem \ref{i2} to $X$ and
get that the tangential variety of $X$ is not defective.
\end{remark}

We also prove the following asymptotic result which is similar to the very general \cite{ah2} and whose proof uses their
differential Horace lemma.

\begin{theorem}\label{i1}
Fix an integral projective variety $X$, line bundles $M, L$ on $X$ with $L$ ample and a zero-dimensional scheme $W\subset X$. There is an integer $k_0$ such that for each integer $k\ge k_0$ and all $(a,b)\in \NN^2$,
either $h^0(\Ii _{Z_{a,b}}\otimes M\otimes L^{\otimes k}) =0$ or $h^1(\Ii _{Z_{a,b}}\otimes M\otimes L^{\otimes k}) =0$, where
$Z_{a,b}\subset X$ is a general union of $a$ 2-points and $b$ tangential schemes. 
\end{theorem}

Since $\sigma _{0,1}\subseteq \sigma _2(X)$,  $\sigma _{a,b}(X)\subseteq \sigma _{a+2b}(X)$. If $\sigma _{a,b}(X)$ and $\sigma _{a+2b}(X)$ have the expected dimension and $\dim \sigma _{a+2b}(X) = (\dim X+1)(a+2b)-1\le N$ the variety $\sigma _{a,b}(X)$ has codimension $b$ in $\sigma _{a+2b}(X)$ and we would like to know its geometry. In particular $\sigma _{a,1}(X)$ seems to be an interesting divisor of the (usually very singular) variety $\sigma _{a+2}(X)$. Let $\sigma ^{00}_{a,b}(X)$ denote the union of all linear spans of schemes $Z$ which are disjoint unions of $a$ 2-points and $b$ tangential schemes of $X_{\reg}$. For each $q\in \sigma ^{00}_{a,b}(X)$ let $\Ss(X,q,a,b)$ denote the set of all schemes $Z\subset X_{\reg}$ such that $q\in \langle Z\rangle$ and $Z$ is a disjoint union of $a$ 2-points and $b$ tangential schemes. An obvious question (identifiability) is to ask if $|\Ss(X,q,a,b)|=1$ for a generic $q\in \sigma _{a,b}(X)$ (generic identifiability, often just called identifiability). One can also ask if $|\Ss(X,q,a,b)|=1$ for a particular $q\in  \sigma ^{00}_{a,b}(X)$ (specific identifiability) and to reconstruct the unique element of $\Ss(X,q,a,b)$.

\subsection{Comments on the results and their proofs}
As far as we know the more recent and powerful versions of the Differential Horace Lemma are \cite{e,r1,r2} which improve the very general appendix of \cite{ah2} (all of them are not limited to the fat point schemes associated to multiple points and they apply to non-complete linear systems). The classical use of the differential Horace lemma is an inductive procedure of the following type. Roughly speaking for any $t\in \NN$ it is defined a certain statement $H_t$, one find a small integer $k_0$ such that $H_{k_0}$ is true (often the most difficult part), one prove that $H_k$ implies $H_{k+1}$ for all $k\ge k_0$ and then one prove that the proven $H_k$'s plus a few small cases prove the big theorem you want to prove. The small cases are important for the full statement, because among them there are the exceptional cases, but they are also important to prove $H_{k_0}$. There is another top-down approach (\cite{ah2,bb,bbcs}, here used to prove Theorem \ref{i1}) which goes like this. Suppose your theorem involves the line bundles $\Ll_t$, $t\in \NN$, all with $h^1(\Ll_t)=0$. There is a family $\Zz_t$ of zero-dimensional schemes such that to prove your theorem for $\Ll_t$ it is sufficient to prove that either $h^1(\Ii_Z\otimes \Ll_t)=0$ or $h^0(\Ii_Z\otimes \Ll_t)=0$ for all $Z\in \Zz_t$. We use the differential Horace lemma with respect to an effective divisor $D$ and see that it is sufficient to prove that either $h^1(\Ii_{Z'}\otimes \Ll_{t-1})=0$ or $h^0(\Ii_{Z'}\otimes \Ll_{t-1})=0$, where $Z'$ is the differential residue. But then (here usually quoting \cite{cm}) we see that there is a smaller $Z_1\subset Z'$ for which it is sufficient to prove that  either $h^1(\Ii_{Z_1}\otimes \Ll_{t-1})=0$ or $h^0(\Ii_{Z_1}\otimes \Ll_{t-1})\le \deg (Z')-\deg(Z_1)$. Then we continue with $Z_1$, not $Z'$, and again applying differential Horace to $Z_1$ we get first some $Z'_1$ and then a smaller $Z_2$ such that to conclude for $Z$ and $\Ll_t$ it is sufficient to prove that either $h^1(\Ii_{Z_2}\otimes \Ll_{t-2})=0$ or $h^0(\Ii_{Z_2}\otimes \Ll_{t-2})\le \deg (Z')-\deg(Z_1)+\deg (Z'_1)-\deg (Z_2)$.  And so on. To win for $(\Ll_t,Z)$ it is sufficient to find some $0\le s<t$ such that either $h^1(\Ii_{Z_s}
\otimes \Ll_{t-s})=0$ or $h^0(\Ii_{Z_s}\otimes \Ll_{t-s}) \le \sum _{i=1}^{s} \deg (Z'_{i-i})-\deg (Z_i)$ with $Z'_0:= Z'$. In many cases (\cite{ah2,bb,bbcs} and Theorem \ref{i1}) using Remark \ref{stu1} to get a very good upper bound for $\deg
(Z)$) we get $Z_s=\emptyset$ and hence $h^1(\Ii_{Z_s}
\otimes \Ll_{t-s})=0$. Thus no need to prove $H_{k_0}$ and the initial cases, we only need to see for which $t$ we are able to find $s\le t$ with small $Z_s$. The exceptional cases may be done, e.g. using a computer, by somebody else and after the asymptotic theory. Contrary to the usual use of the differential Horace doing first the exceptional cases does not help to get the asymptotic result.

In \cite{ah2,bb,bbcs} it was clear that to use mixed multiplicity drastically simplify the numerology (see here
Remark
\ref{stu1} for a tiny example of the simplification which can be obtained using the mixed case, even if one is only interested
in the case $a=0$).

In this paper there are $3$ different blocks of statements.

\quad (i) Theorem \ref{uu1} and Propositions \ref{ux2}, \ref{ux3}. This case use the classical differential Horace lemma inductive proof, because in this case it works very well. All exceptional pure cases $a=0$ or $b=0$ where known (\cite{av,ah2}) and the only exceptional case with $ab\ne 0$ is geometrically obvious and with small degree and very low
$n$. The case $n=2$ (as all cases with an arbitrary integral surface $X$ as a reader of \cite{bf} may realize) are very easy
if the induction may start. The cases $n=3,4$ (with a few exceptional cases with $b=0$ for low $d$) are easy, but using the
explicit values of the integers $\binom{n+d}{n}$ for low $n$ and low $d$ (Propositions \ref{ux2} and \ref{ux3}). Then the inductive proof
by $n$ is routine. The interested reader may improve Theorem \ref{uu1} taking both $a$ and $b$ a bit smaller. Here examples done using a computer
should be very useful to see what is reasonable to expect.

\quad (ii) Theorem \ref{i2}, Corollaries \ref{i2.0}, \ref{i2.1} and Remark \ref{nni2.3}. This case is an hybrid: its proof is stated as the classical differential Horace inductive proof, but its proof could be written top down.
In the corollaries $C =\PP^1$, but we stress that the main theorem and the main idea is for an arbitrary curve $C$. We fix a
general $o\in C$ and we use the differential Horace lemma with respect to the divisor $Y\times \{o\}$. The interested reader
consider the following problem. 

Fix integral projective curves $C_1,\dots ,C_k$ of arithmetic genus $g_1,\dots ,g_k$. Find
positive integers $a_1,\dots ,a_k$ (only depending on $g_1,\dots ,g_k$ and, say, $a_i\ge 2g_i+1$ for all $i$) such that for all
integers $d_1,\dots ,d_k$ with $d_i\ge a_i$ and any line
 bundle $L_i$ of degree $d_i$ on $C_i$ the embedding of $C_1\times \cdots \times C_k$ by the complete linear system $H^0(L_1)\boxtimes \cdots \boxtimes H^0(L_k)$ has the property that its tangential variety is not defective. To start the induction on the number $k$ of factors one needs the initial cases and hence one need to first give a proof for $C_1\times C_2$. The case of an integral surface $X$ is easy (if we know the initial cases) as hinted in \cite{bf}. In \cite{bb} and \cite{bbcs} we used from the literature this case for multiple points with multiplicities up to $4$ and up to $5$ respectively, and then we considered the case $\dim X=3$ and $\dim D=2$ (see \cite[Lemma 7]{bbcs}, which is the equivalent in that set-up of Remark \ref{stu1}). In all cases it is clear that the use of mixed multiplicity should drastically simplify the computations, even the one done by computer in \cite{bbcs}. Since $p_a(C_i)$ is allowed  to be positive, we may arrive going down to a line bundle $\Ll$ on $X$ such that $h^0(\Ll)\ne 0$ and $h^1(\Ll)>0$. Since $h^0(\Ll )-h^0(\Ii_A\otimes \Ll) = h^1(\Ii_A\otimes \Ll)-h^1(\Ll)$ for any zero-dimensional scheme $Z$, $Z$ gives $\deg (Z)$ independent conditions to $H^0(\Ll)$ if and only if $h^1(\Ii_Z\otimes \Ll)=h^1(\Ll)$.
 
\quad (iii) Theorem \ref{i1}. We believe that the theorem proved by J. Alexander and A. Hirschowitz in \cite{ah2}, i.e. the
case of Theorem \ref{i1} without tangential schemes, is very strong and interesting, not just a consolation prize for not
having a complete classification. To prove Theorem \ref{i1} we will follow their strategy (e.g. induction on $n$) and at the
beginning of the proof we will explain a critical reduction step. We point out that even to prove it only for tangential
schemes in the intermediate steps many 2-points occur and hence it is easier to prove the mixed case. The paper \cite{ah2} is non
just for 2-points but for points with arbitrary multiplicities, not necessary the same, with the only restriction that in the statement of the theorem it is prescribed a positive integer $\mu$ and all multiple points have multiplicities $\le \mu$. Even if at the beginning we claim a statement only for a general union of multiple points with the same
multiplicity $\mu >2$ in the intermediate steps mixed multiplicities occur. 

Fix
$p\in X_{\reg}$. The scheme $mp$ is defined by the isomorphism class of the pair $(X,p)$. A tangential scheme $Z$ with $p$
as its reduction is uniquely determined by the isomorphism class of a triple $(X,p,R)$, where $R$ is a one-dimensional vector
space of the tangent space of $X$ at $p$. It is this intrinsic definition which makes possible statements and proofs for any $X$.
Similarly, the tangent space (and hence the dimension) of osculating varieties is associated to zero-dimensional schemes
uniquely determined by the isomorphism class of a triple $(X,p,R)$, plus the order of osculation (\cite{bf}). Thus it is reasonable to
suppose that Theorem \ref{i1} may be extended to the case of general unions of multiple points and osculating schemes, if we
fix the maximum order of the multiple points and of the osculating points. Instead of $(X,p,R)$ one could use $(X,p,\Ff)$, where $\Ff$ is a partial flag of the vector space $T_pX$.

\subsection{Conclusions and further suggestions}
We described the top-down way to use the general Horace method, without the need of initial cases to start the induction. Thus it can be done even in
a case (e.g. for Theorem \ref{i2} for $C$ with positive arithmetic genus $g$) where our feeling is that it is not possible to check by
computer some cases, each of them depending on parameters ($4g-3$ parameters for $g\ge 2$, $2$ for $g=1$). But even in
the case, like homogeneous varieties embedded by a homogeneous complete linear system, in which we need to check only finitely
many cases
 the top-down approach as the advantage that the cases may be checked by another team after the theoretical part is done.
For
the theoretical part the mixed case (allowing 2-points) seems to simplify drastically the proof and we think that it may help
 to cut the cases to be
verified by the computer (as here Remark \ref{stu1}, in \cite{bbcs}, in particular using \cite[Lemma 7]{bbcs}). We also asked
the identifiability problem for a general $q\in \sigma _{a,b}(X)$ and for specific
$q\in \sigma ^{00}_{a,b}(X)$. In (ii) and (iii) we raised other open questions.

We thank a referee for stimulating observations.

\section{Preliminary lemmas}
 For any integral projective variety $X$, any $L\in \mathrm{Pic}(X)$, any linear subspace $V\subseteq H^0(L)$ and any closed
subscheme $B\subset X$ set $V(-B):= H^0(\Ii _B\otimes L)\cap V$. Note that $\dim V$ means the dimension of $V$ as a vector
space and so if $V\ne 0$ the linear system on $X$ associated to $V$ has dimension $\dim V-1$. Obviously $\dim V(-B)= \max
\{0,\dim V -|B|\}$ if (for a fixed $V$) we take a general subset $B \subset X$ with a prescribed number of points. We say that
a scheme $A\subset X$ is a \emph{tangent vector} of $X$ if $A$ is a connected zero-dimensional scheme, $\deg (A)=2$ and
$A\subset X_{\mathrm{reg}}$. Since $X$ is integral and we impose that the support of a tangent vector is a smooth point of
$X$, the set of all tangent vectors of $X$ is an irreducible quasi-projective variety of dimension $2(\dim X)-1$. Moreover, for
any integral subvariety $T\subset X$ with $T\nsubseteq \mathrm{Sing}(X)$ and $\dim T>0$ the set of all tangent vectors of $X$
whose reduction is a point of $T$ (resp. which are tangent vectors of $T_{\reg}$) is an irreducible variety of dimension $\dim T +\dim
X -1$ (resp. $2(\dim T)-1$).

For any projective scheme $X$, any effective Cartier divisor $D$ of $X$ and any closed subscheme $Z\subset X$ the residual
scheme $\Res _D(Z)$ of $Z$ with respect to $D$ is the closed subscheme of $X$ with $\Ii _Z:\Ii _D$ as its ideal sheaf. For any
$\Ll\in \mathrm{Pic}(X)$ we have the residual exact sequence of coherent sheaves on $X$:
\begin{equation}\label{eq++ss10}
0 \to \Ii _{\Res _D(Z)}\otimes \Ll (-D) \to \Ii _Z\otimes \Ll \to \Ii _{Z\cap D}\otimes \Ll _{|D}\to 0\end{equation}
If $Z$ is zero-dimensional, then  $\deg (Z) = \deg (Z\cap D)+\deg (\Res _D(Z))$. 

\begin{remark}\label{a1}
Let $X$ be an integral projective variety of dimension $\ge 2$ and let $D\subset X$ be an integral effective divisor, $D\ne
\emptyset$. Fix $\Ll \in \mathrm{Pic}(X)$ and a vector space $V\subseteq H^0(\Ll)$. For a general finite set $S\subset D$
we have $\dim V(-S) = \max \{\dim V(-D), \dim V -|S|\}$.
\end{remark}

Since we only work in characteristic zero, the following lemma is a particular case of \cite{cm}.

\begin{lemma}\label{a2}
Let $X$ be an integral projective variety of positive dimension, $L\in \mathrm{Pic}(X)$ and $V\subseteq H^0(L)$ a linear
subspace. Fix $(x,y)\in \NN^2$. Let $Z\subset X$ be a general union of $x$ points and $y$ tangent vectors of $X$. Then $\dim
V(-Z)=\max \{0,\dim V -x-2y\}$.
\end{lemma}

As in most papers using the Horace method we use several times the following obvious observation. 

\begin{remark}\label{stupid}
Let $X$ be a projective scheme, $L$ a line bundle on $X$ and $W \subseteq W'$
zero-dimensional subchemes of $X$. To prove that $h^0(\Ii_{W'}\otimes L)=0$ it is sufficient to prove that $h^0(\Ii_W\otimes
L)=0$.
To prove that $h^1(\Ii_W\otimes L)=0$ it is sufficient to prove that $h^1(\Ii_{W'}\otimes L)=0$
\end{remark}

\begin{lemma}\label{a3}
Let $X$ be an integral projective variety with $\dim X\ge 2$, $D$ an integral divisor on $X$, $D\ne \emptyset$, $L$ a line
bundle on $X$ and $V\subseteq H^0(L)$ a linear subspace. Fix $(x,y)\in \NN^2$. Let $Z\subset D$ be a general union of $x$
points of $D$ and $y$ tangent vectors of $D$. Let $V_{|D}$ be the image of $V$ in $H^0(D,L_{|D})$. Then:
\begin{enumerate}
\item $\dim V_{|D} =
\dim V - \dim V(-D)$ and $\dim V -
V(-Z) =
\min \{\dim V_{|D},x+2y\}$. 
\item We have $\dim V(-Z) =\dim V -x-2y$ (resp. $V(-Z)=V(-D)$) if $\dim V_{|D} \ge x+2y$ (resp. $\dim V_{|D} \le x+2y$).
\item If $V = W(-E)$ for some zero-dimensional scheme $E\subset X$ and $W \subseteq H^0(L)$ is a linear subspace, thene $V_{|D} = W_{|D}(-E\cap D)$ and $V(-D) =W(-\Res _D(E))$.
\end{enumerate}
\end{lemma}

\begin{proof}
The equality $\dim V_{|D} =
\dim V - \dim V(-D)$ is obvious and it implies the second assertion of (1) by Lemma \ref{a2} applied to $D$.
Part (2) is a particular case of part (1). Part (1) follows from the residual exact sequence \eqref{eq++ss10}, because $V(-Z) =V\cap H^0(\Ii_Z\otimes L)$.
\end{proof}

\begin{remark}\label{ff0}
Let $V$ be a linear system on an integral projective curve $Y$. We claim that the join of an arbitrary number of copies of the tangential variety of the embedded curve
$Y'$ associated to $V$ and an arbitrary number of copies of $Y'$ has the expected dimension. Since we are in characteristic
zero, to prove this claim it is sufficient to prove that for any general finite set
$S =S_1\sqcup S_2\subset Y_{\mathrm{reg}}$ we have $\dim V(-Z) =\max \{0,\dim V - 3|S_1|-2|S_2|\}$, where $Z$ is the effective
Cartier divisor of
$Y$ with each point of $S_1$ appearing with multiplicity $3$ and each point of $S_2$ appearing with multiplicity $2$. This
is true by \cite{cm}.
\end{remark}

\begin{remark}\label{ff001}
Let $X$ be an integral projective variety, $L\in \mathrm{Pic}(X)$ and $V\subseteq H^0(L)$ a linear subspace such that the
rational map $f: X\dasharrow \PP^r$, $r:= \dim V-1$, has image of dimension $n:= \dim (X)$. Take a non-empty open subset $U$
of $X$ such that the rational map $f$ defined by $V$ is a morphism on $U$ and $f_{|U}$ has injective differential. Call
$X'\subset
\PP^r$ the closure of $f(U)$
and $V' \subseteq H^0(\Oo _{X'}(1))$ the image of the restriction map $H^0(\Oo _{\PP ^r}(1)) \to H^0(\Oo _{X'}(1))$.
We have $\dim V' = \dim V$. By Chevalley's theorem (\cite[Exercises II.3.18, II.3.19]{h}) there is a non-empty open subset $U'\subset X'$ such that $f(U)\supset
U'$. Fix
$a\in
\NN$. Let
$Z'\subset X'$ be a general union of
$a$ tangential schemes of
$X$. Since
$Z'$ is general, $Z'\subset U'$ and there is a (not uniquely determined) scheme $Z\subset U$ such that $f_{|U}$ induces an
isomorphism $Z\to Z'$. Note that $\dim V(-Z) = \dim V'(-Z')$. Thus we may check if the tangential variety of $X'$ with respect
to $V'$ and its secant varieties have the expected dimension using $X$ and $V$, although  we do not have defined the tangential variety of the pair $(X,V)$. The definition of a
tangential scheme, $Z_1$, of $X$ only requires a point $o\in X_{\mathrm{reg}}\cap U$ and a one-dimensional linear subspace of
the tangent space of $X$ at $o$, because the differential of $f$ is injective at $o$.
\end{remark}

Let $X\subset \PP^r$ be an integral and non-degenerate $n$-dimensional variety. For any $q\in X_{\mathrm{reg}}$ the Zariski
tangent space $T_qX$ of $X$ at $q$ is the linear span of the degree $n+1$ zero-dimensional scheme $2q\subset X$ with $(\Ii
_q)^2$ as its ideal sheaf; we say that $2q$ is the 2-point with the point $q$ as its reduction. We have $(2q)_{\red} =\{q\}$.
Fix
$q\in  X_{\mathrm{reg}}$ such that
$T_qX\nsubseteq
\mathrm{Sing}(\tau (X))$ (e.g. take a general $q\in X)$. Fix any $p, p'\in T_qX\cap \tau (X)_{\mathrm{reg}}$ such that the line
spanned by $\{q,p\}$ contains $p'$. We have $T_p\tau (X) = T_{p'}\tau (X)$ and these linear subspaces of $\PP^r$ are the
linear span of a certain degree $2n+1$ zero-dimensional scheme constructed using the tangent vector of $X$ whose linear span
is the line containing $\{p,q\}$.

Now we fix an effective divisor $D$ of $X$ and $q\in D_{\mathrm{reg}}\cap X_{\mathrm{reg}}$. We write $2q$ for the $2$-point
of $X$ with $q$ as its reduction and $\{2q,D\}$ for the 2-point of $D$ with $q$ as its reduction. We have
$2q\cap D =\{2q,D\}$ (as schemes) and $\Res _D(2q) = \{q\}$. We  say that $2q$ has type $(n,1)$ with respect to $D$, because
$\deg (D\cap 2q)=n$, $\deg (D\cap \Res _D(2q)) =1$ and $\Res _D(\Res _D(2q)) =\emptyset$, but this is only a shorthand for a
stronger information: not only $D\cap 2q$ has degree $n$, but it is $\{2q,D\}$.

 Fix a closed subscheme
$W\subset X$ with
$q\notin W_{\red}$. A key tool introduced by J. Alexander and A. Hirschowitz (\cite{ah, ah2}) is that to prove that $W\cup A$
has a good cohomological property (e.g. for
$X\subset
\PP^r$ that the linear span of $W\cup A$ has the expected dimension) it is sufficient to prove it for a virtual scheme $A_q$
with
$A_q\cap D = \{q\}$ and $\Res _D(A_q) = \{2q,D\}$. We say that the virtual scheme $A_q$ has type $(1,n)$ with respect to $D$,
because its intersection with $D$ has degree $1$ and its residue with respect to $D$ has degree $n$ and it is contained in
$D$; again this is only a shorthand, because we use in an essential way that the residual scheme is $\{2q,D\}$.

Let
$Z\subset X$ be the degree
$2n+1$ zero-dimensional scheme associated to a tangent vector $\mathbf{v}$ which we called a \emph{tangential scheme}. If
$\mathbf{v}\subset D$, then $\deg (D\cap Z)=2n-1$, $\deg (\Res _D(Z)) = 2$ and $\Res _D(Z)\subset D$. We say that $Z$ has type
$(2n-1,2)$ with respect to $D$. This is only a shorthand, because we use that $Z\cap D$ is the tangential scheme of $D$
associated to the tangent vector $\mathbf{v}$. If $\mathbf{v}_{\red}\in D_{\mathrm{reg}}$ and $\mathbf{v}\nsubseteq D$, then $Z\cap D =\{2q,D\}$, $\Res _D(Z)\cap D
=\{2q,D\}$ and $\Res _D(\Res _D(Z)) = \{q\}$. In this case we say that $Z$ has type $(n,n,1)$ with respect to $D$. A key
observation in \cite{cgg} is that we may extend the differential Horace and see that to handle the Hilbert function of $W\cup
Z_1$ for a general tangential scheme $Z_1$ of $X$ it is sufficient to handle the Hilbert function of $W\cup Z'$ with $Z'$ a
virtual scheme with $Z'\cap D$ a tangent vector $\mathbf{v}$ of $D$ at $q$ and $\Res _D(Z')$ the tangential scheme of $D$
associated to $\mathbf {v}$. We say that $Z'$ has type $(2,2n-1)$ with respect to $D$.

\begin{remark}\label{stu1}
Let $X$ be an integral projective variety and $\Ll$ a line bundle on $X$ such that $h^1(\Ll)=0$. Set $n:= \dim X$. For all
$(a,b)\in
\NN^2$ let
$Z_{a,b}$ denote a general union of $a$ 2-points and $b$ tangential schemes. Suppose you want to prove that for each $(a,b)\in
\NN^2$ either $h^0(\Ii_{Z_{a,b}}\otimes \Ll)=0$ or $h^1(\Ii_{Z_{a,b}}\otimes
\Ll)=0$. By Remark \ref{stupid} it is sufficient to test all $(a,b)\in \NN$ such that
\begin{equation}\label{eqstu1}
h^0(\Ll)-n\le a(n+1)+b(2n+1) \le h^0(\Ll)+2n.
\end{equation}However, we can do better. Fix $(a,b)\in \NN^2$
satisfying \eqref{eqstu1}. Since $h^1(\Ll) =0$, the case $(a,b)=(0,0)$ need not be tested. First assume $b>0$ and
$a(n+1)+b(2n+1) \ge h^0(\Ll)+n$. Since a general tangent scheme contains a general 2-point, to test $(a,b)$ it is
sufficient to test $(a+1,b-1)$. If
$b=0$ and
$a(n+1)
\ge h^0(\Ll) +n+1$ it is sufficient to test $(a-1,0)$. Thus it is sufficient to test all $(a,b)$ such that either $b=0$ and
$h^0(\Ll)-n\le a(n+1)\le h^0(\Ll)+n$ or
\begin{equation}\label{eqstu2}
b>0,\ h^0(\Ll)-n\le a(n+1)+b(2n+1) \le h^0(\Ll)+n-1.
\end{equation}
If $a(n+1)+b(2n+1) =h^0(\Ll)-n$ and $a>0$, we may avoid to test $(a,b)$ if we test the more difficult case $(a-1,b+1)$. Conversely, if
$a>0$, $a(n+1)+b(2n+1) =h^0(\Ll)-n$ and $(a,b)$ is an exeptional case, then $(a-1,b+1)$ is an exceptional case.

With minimal modifications this observation may be extended to a case in which there is a small finite set $\Sigma \subset
\NN^2$ such that for all $(a,b)\in \Sigma$ we do not know if $h^0(\Ii_{Z_{a,b}}\otimes \Ll)\cdot h^1(\Ii_{Z_{a,b}}\otimes
\Ll)=0$.
\end{remark}

\section{Proof of Theorem \ref{i2} and its corollaries}\label{Sy}

Let $Y$ be an integral projective variety and $C$ an integral projective curve. Set
$n:=
\dim (Y)+1$ and $X:= Y\times C$. Call $\pi _1: X\to Y$ and $\pi _2: X\to C$ the projections. For all $L\in \mathrm{Pic}(Y)$
and $R\in \mathrm{Pic}(C)$ set $L\boxtimes R:= \pi _1^\ast (L)\otimes \pi _2^\ast ({R})\in  \mathrm{Pic}(X)$. K\"{u}nneth gives
$H^0(L\boxtimes R) \cong H^0(L)\otimes H^0({R})$. Thus for all linear subspaces $V\subseteq H^0(L)$, $W\subseteq H^0({R})$,
$V\boxtimes W:= \pi_1^\ast (V)\otimes \pi_2^\ast (W)$ is a linear subspace of $H^0(L\boxtimes R)$ with $\dim V\boxtimes W =
(\dim V)(\dim W)$. Taking instead of $Y$ its image by the linear
system associated to $V$ (Remark \ref{ff001}) we may assume that $V$ induces an embedding of $Y$ into $\PP^{\alpha -1}$, $\alpha := \dim V$.  Fix a general $o\in C$. Since we are in characteristic zero and $o$ is general, $o$ is not an osculating
point
of $C$ with respect to the linear system associated to $W$. Thus $\dim W(-xo) =\max \{0,\dim W-x\}$ for all $x\in \NN$. Set
$D:= Y\times \{o\}$. For each $t\in \NN$ set $V[t]:= V\boxtimes W(-(\dim W -t-1)o)$. We have $\dim V[t] =(t+1)\cdot \dim V$ and
$V[0]\cong V$ up to the identification of $Y$ with $D$ given by $\pi _{1|D}$. 
 
\begin{notation}For each $a\in \NN$ let $Z_a$ (resp. $W_a$) denote a general union of $a$ tangential schemes of $X$ (resp. $D$). Set $\alpha :=
\dim V$ and $a_1:= \lfloor \alpha/(2n-1)\rfloor$. Write $\alpha = (2n-1)a_1+\mu$ with $0\le \mu \le 2n-2$. Write $\alpha = (2n-1)f_1 + 2e_1$ with $f_1 =a_1$ and $e_1=\mu /2$ if $\mu $ is even and $f_1 =a_1-1$ and $e_1 = n+(\mu -1)/2$ if $\mu$ is odd. Since $\mu \le 2n-2$, we have $e_1\le 2n-2$. The pair $(f_1,e_1)$ is the only pair $(x,y)\in \ZZ^2$ such that $(2n-1)x+2y =\alpha$ and $0\le y\le 2n-2$.
\end{notation}

We make the
following assumption
$\pounds$:

\quad \textbf{Assumption} $\pounds$: For all $a\in \NN$ the vector space $V(-W_a)$ has the expected dimension $\max \{0,\alpha -(2n-1)a\}$.

By Assumption $\pounds$ we have $\dim V(-W_a) =\alpha -a(2n-1)$ if $a \le a_1$ and $V(-W_a) = 0$ if $a>a_1$.

\begin{lemma}\label{y1}
We have $\dim V[0](-Z_a) = \max \{0,\alpha -(2n-1)a\}$.
\end{lemma}

\begin{proof}
Take any tangential scheme $Z\subset X$. The scheme $\pi _1(Z)$ has degree at most $2n-1$ and it has degree $2n-1$ if and only
if the tangent vector $v$ used to define $Z$ is not a tangent vector of the fiber of $\pi _2$ at the point $\pi
_2(Z_{\mathrm{red}})$, i.e. if and only if $\pi _1(Z)$ is a tangential scheme of $Y$. Use $\pounds$.
\end{proof}

\begin{lemma}\label{y2}
Set $x_1:= \lfloor \alpha /(2n+1)\rfloor$ and $\tau _1:= \alpha -(2n+1)x_1$.

\quad (a) We have $\dim V[1](-Z_a) = 2\alpha -(2n+1)a$ for all $a\le 2x_1$.

\quad (b) If $a \ge 2\lceil \alpha /(2n+1)\rceil$, then $V[1](-Z_a)=0$.

\quad ({c}) If $(2n+1)a \ge 2\alpha +2n+2$, then $V[1](-Z_a) =0$.

\quad (d) We have $\dim V[1](-Z_{2x_1+1})\le \max \{0,\tau _1-2\}$.

\end{lemma}

\begin{proof}
By Remark \ref{stupid} to prove part (a) it is sufficient to do the case $a=2x_1$. We specialize $Z_{2x_1}$ to a general $W_1:= W\sqcup
W'$ with
$W$ a general union of $x_1$ schemes of type $(2n-1,2)$ with respect to $D$ and $W'$ a general union of $x_1$ virtual schemes
of type $(2,2n-1)$ with respect to $D$. Since $D\cap W_1$ is a general union of $x_1$ tangential schemes of $D$ and $x_1$
tangent vector of $D$ and $\alpha \ge (2n+1)x_1$, the assumption $\pounds$ and Remark \ref{a1} give $\dim V[0](-W_1\cap D)
=\alpha -\deg (W_1\cap D)$. Since $\Res _D(W_1)$ is a general union of $x_1$ tangent vectors of $D$ and $x_1$ tangential
schemes of
$D$, we get part (a).

To get part (b) instead of $W_1$ we use a general union of $\lceil \alpha /(2n+1)\rceil$ schemes of type $(2n-1,2)$ with respect to $D$ and a
general union of $\lceil \alpha /(2n+1)\rceil$ schemes of type $(2,2n-1)$ with respect to $D$. Note that $x_1+1\ge \lceil \alpha /(2n+1)\rceil$. Thus by part (b) to check part ({c}) it is sufficient to
observe
that $(2x_1+1)(2n+1) \le 2\alpha +2n+1$ and hence the assumption of ({c}) gives $a\ge 2x_1+2$.

To get part (d) we specialize $Z_{2x_1+1}$ to a general union of $x_1+1$ schemes of type $(2n-1,2)$ with respect to $D$ and $x_1$ schemes of type $(2,2n-1)$ with respect to $D$.\end{proof}

\begin{lemma}\label{y2+=}
Assume $n\ge 3$ and $\alpha \ge (n+1)(2n-1)$; if $n=3$ and $\mu >0$ assume $a_1\ge 9-2\mu$.

\quad (a) If $(2n+1)a \le 3\alpha -\mu$, then $\dim V[2](-Z_a) =3\alpha - (2n+1)a$.

\quad (b) If $(2n+1)a\ge  3\alpha +2n-1-\mu$,  then $\dim V[2](-Z_a) =0$.

\end{lemma}

\begin{proof}
The assumption $\alpha \ge (n+1)(2n-1)$ is equivalent to $a_1\ge n+1$. 

\quad (a) To prove part (a) it is sufficient to do the case $a= \lfloor \frac{3\alpha -\mu}{2n+1}\rfloor$ and in
particular we may assume $a\ge a_1$. We specialize
$Z_a$ to
$W_1:= Z_{a-a_1}\cup W$, where
$W$ is a general union of $a_1$ schemes of type $(2n-1,2)$ with respect to $D$. By $\pounds$ we have $\dim V[0](-W\cap D) =
\alpha -(2n-1)a_1$. Thus it is sufficient to prove that $V[1](-\Res _D(W_1))$ has the expected dimension. We have $\Res
_D(W_1) = Z_{a-a_1}\cup \Res _D(W)$ with $\Res _D(W)$ a general union of $a_1$ tangent vectors of $D$. Note that
$(2n+1)(a-a_1) \le 3\alpha -\mu -\alpha +\mu -2a_1 = 2\alpha -2a_1 $. Since $a_1\ge n+1$, we have $(2n+1)(a-a_1) \le 2\alpha -2n-2$.
Thus $a-a_1$ is not the unique integer, $2x_1+1$, which is not covered either in part (a) or in part (b) of Lemma \ref{y2}.
Since $(2n+1)(a-a_1) \le 2\alpha$, parts (a) and (b) of Lemma \ref{y2} give that $V[1](-Z_{a-a_1})$ has the expected
dimension, $2\alpha -(2n+1)(a-a_1)$. Since $D\cap Z_{a-a_1}=\emptyset$ and $\deg (\Res _D(W)) =2a_1\le \alpha$, to prove part (a) using Lemma \ref{a3} it is sufficient to check that $2a_1\le \alpha $
(true because $2n-1 \ge 2$ for $n\ge 3$) and that $\dim V[0](-Z_{a-a_1})\le \max \{0,2\alpha -(2n+1)(a-a_1) -2a_1\}$. By
$\pounds$ we  have
$\dim V[0](-Z_{a-a_1}) = \max \{0,\alpha -(2n-1)(a-a_1)\}$. Thus we may assume $a-a_1 \le a_1$. In this case we have $\dim
V[0](-Z_{a-a_1}) =\alpha -(2n-1)(a-a_1)$. Thus we conclude, unless $a\le 2a_1$ and $2\alpha -(2n+1)(a-a_1) -2a_1 \le -1
+\alpha -(2n-1)(a-a_1)$, i.e. $a\le 2a_1$ and $\alpha +1 \le 2(a-a_1) +2a_1 =2a$ and in particular $\alpha \le 4a_1-1$. We get
$(2n-1) < 4$ and hence $n=2$.

\quad (b) Now we prove part (b). By Remark \ref{stupid} it is sufficient to do the case $a =\lceil \frac{3\alpha +2n-1 -\mu}{2n+1}\rceil$. We
specialize
$Z_a$ to
$W_2:= Z_{a-a_1-1}\cup W'$, where
$W'$ is a general union of $a_1+1$ schemes of type $(2n-1,2)$ with respect to $D$. By $\pounds$ we have $V[0](-W\cap D)=0$ and so it
is sufficient to prove that $V[1](-\Res _D(W_2)) =0$. By Lemma \ref{a3} it is sufficient to prove that
$V[1](-Z_{a-a_1-1})$ has the expected dimension and that $V[0](-Z_{a-a_1-1}) =0$. For $V[1]$ we only need to check that
$a-a_1-1 \le 2x_1$ (Lemma \ref{y2}), i.e. $(2n+1)(a-a_1-1) \le 2\alpha$. Since $(2n+1)a\le 3\alpha +4n-1-\mu$, to check $V[1]$
we only need that $3\alpha +4n-1-\mu -2a_1-\alpha +\mu -2n-1 \le 2\alpha$, i.e. it is sufficient to have $n+1 \le a_1$. 
For $V[0]$ we need that either $\mu =0$ and $a-a_1\ge a_1$ or $a-a_1-1\ge a_1+1$. Since $a\ge ((6n-3)a_1+2\mu +2n-1)/(2n+1)$, we have $a\ge 2a_1+2$
if $(4n+2)a_1 + 4n+2 \le (6n-3)a_1+2\mu +2n-1$, which is true if $n\ge 4$ and $a_1\ge n+1$ or $n=3$ and $a_1\ge 9-2\mu$.
\end{proof}

\begin{lemma}\label{ny1}
Assume $n\ge 3$ and $\alpha \ge (n+1)(2n-1)$; if $n=3$ and $\mu >0$ assume $a_1\ge 9-2\mu$. Fix $z\in \NN$ and let $Z =Z_z\cup W\subset X$ with $W$ a general union of $e_1$ tangential schemes of $D$. Assume $(2n+1)z +(2n-1)e_1 \le 3\alpha -\epsilon$ with $\epsilon =3$ if $\mu =2$ and $\epsilon =0$ otherwise. Then $\dim V[2](-Z) =3\alpha -\deg (Z)$.
\end{lemma}

\begin{proof}
Note that $\mu =0$ if and only if $e_1=0$ and that $\mu =2$ if and only if $e_1=1$. By Lemma \ref{y2+=} we may assume $\mu \ne 0$. By Remark \ref{stupid} we may assume $x\ge f_1$. To prove part (a) we specialize $Z_z$ to $Z_{z-f_1}\cup A$. Since $\deg (D\cap (A\cup W)) =\alpha$, $V[0](- D\cap (A\cup  W))=0$ and $\deg (\Res _D(A)) \le \alpha$, by Lemma \ref{a3} it is sufficient to prove that $V[1](-Z_{z-f_1}) = 2\alpha -(2n+1)(z-f_1)$ and that $\dim V[0](-Z_{z-f_1}) 
\le \max \{0, 2\alpha -\deg (Z_{z-f_1}\cup \Res _D(A))\}$. By Lemma \ref{y2} to prove that $V[1](-Z_{z-f_1}) = 2\alpha -(2n+1)(z-f_1)$ it is sufficient to prove that
$z-f_1 \le 2\lfloor \alpha /(2n+1)\rfloor$. Since $(2n+1) +2(2n+1) \lfloor \alpha /(2n+1)\rfloor \ge (2n+1) +2(2n+1) (\alpha -2n)/(2n+1) =2\alpha -2n$,  it is sufficient to prove that $(2n+1)z -(2n+1)f_1  \le 2\alpha -2n$, i.e. $(2n+1)z \le (6n-3)f_1 +4e_1-2n$.
By assumption we have $(2n+1)z +(2n-1)e_1 \le (6n-3)f_1 +6e_1-\epsilon$, i.e. $(2n+1)z \le (6n-3)f_1+(7-2n)e_1  -\epsilon$. We have $(2n-3)e_1 \ge 2n-\epsilon$ because either $e_1\ge 2$ or $e_1=1 =\mu /2$.

We have $V[0](-Z_{z-f_1}) $ if either $z-f_1\ge a_1+1$ or $\mu =0$ and $z-f_1 \ge a_1$. However, if $z-f_1\le a_1$, we have $\dim  V[0](-Z_{z-f_1}) = \alpha -(2n-1)(z-f_1)$
and so to handle $V[0]$ we may assume $z-f_1 \le a_1$ and use that  $\alpha \ge 2(z-f_1) + 2f_1$ (this is true, because $n\ge 3$). 
\end{proof}

\begin{lemma}\label{y3}
Assume $n \ge 3$ and $\alpha \ge 4n^2  +2n -4$. For any
$a\in
\NN$
$Z_a$ has the expected dimension with respect to
$V[3]$.
\end{lemma}
\begin{proof}
By Remark \ref{stupid} it is sufficient to do the cases $a\in \{\lfloor 4\alpha/(2n+1)\rfloor ,\lceil 4\alpha/(2n+1)\rceil \}$ and in
particular we may assume $f_1+e_1 \le a \le (4[(2n-1)f_1+2e_1]+2n)/(2n+1)$. We specialize
$Z_a$ to
$Z_{a-f_1-e_1}\cup W\cup W'$ with $W$ a general union of $f_1$ schemes of type $(2n-1,2)$ with respect to $D$ and $e_1$
schemes of type
$(2,2n-1)$ with respect to $D$. By $\pounds$ and Lemma \ref{a2} it is sufficient to prove that $Z_{a-f_1-e_1}\cup \Res _D(W)
\cup \Res _D(W')$ imposes independent conditions to $V[2]$. Since $\deg (D\cap (Z_{a-f_1-e_1}\cup \Res _D(W)
\cup \Res _D(W')))\le \alpha$, by Lemma \ref{a3} it is sufficient to prove that
$\dim V[2](-(Z_{a-e_1-f_1}\cup \Res _D(W'))) = 3\alpha -(2n+1)(a-e_1-f_1) -(2n-1)e_1$ and that $\dim V[1](-Z_{a-e_1-f_1}) \le
\max
\{0,3\alpha -(2n+1)(a-e_1-f_1) -2e_1
\}$. To prove the assertion about
$V[2]$  using Lemma \ref{ny1} it is sufficient to prove that $\deg (Z_{a-f_1-e_1}\cup \Res _D(W)) \le 3\alpha -\epsilon$. Since $Z_{a-f_1-e_1}\cup \Res _D(W)$ has degree
$(2n+1)a +\alpha -2f_1$ when $\epsilon =0$ it is sufficient to observe that $(2a+1)a\le 4\alpha +2n$ and that $f_1\ge n$. Now assume $\epsilon =3$, i.e. $\mu =2$ and $e_1=1$.
In this case $f_1=a_1 =(\alpha -2)/(2n-1) \ge n+3$.

Now we consider $V[1]$. Since $\dim V[2]-\dim V[1] =\alpha$, to conclude by Lemma \ref{a3} it is sufficient to prove that $\dim V[1](-Z_{a-e_1-f_1}) \le \dim V[2] - \deg (Z_{a-f_1-e_1}\cup \Res _D(W)
\cup \Res _D(W'))$. We have $\deg (Z_{a-f_1-e_1}\cup \Res _D(W)
\cup \Res _D(W')) =(2n+1)a -\alpha$. We are done if either $V[1](-Z_{a-e_1-f_1}) =0$ or $\dim V[1](-Z_{a-e_1-f_1}) = 2\alpha \deg (Z_{a-e_1-f_1})$. Thus by Lemma \ref{y2}
we may assume $a-e_1-f_1 = 2x_1+1$ and $\dim V[1](-Z_{a-e_1-f_1}) \le \max \{0,\tau _1-2\}$. It is sufficient to have $\alpha \ge \max \{0,\tau _1-2\} + 2f_1+(2n-1)e_1$,
i.e. $(2n-3)f_1\ge (2n-3)e_1+ \max \{0,\tau _1-2\} $. Since $\tau _1 \le 2n$, it is sufficient to have $(2n-3)f_1\ge (2n-3)e_1 +2n-2$. Since $n\ge 3$, it is sufficient to have $f_1\ge e_1+2$. Since $e_1\le 2n-2$, it is sufficient to assume $\alpha \ge (2n-1)2n +2(2n-2) = 4n^2  +2n -4$. \end{proof}

\begin{lemma}\label{y4}
Take the assumptions of Lemma \ref{y3}. Let $B\subset D$ be a general union of $e_1$ tangential schemes of $D$. For any $a\in \NN$ $Z_a\cup B$ gives the expected number of conditions to $V[3]$.
\end{lemma}

\begin{proof}
First assume $(2n+1)a +(2n-1)e_1\le 4\alpha$. By Remark \ref{stupid} we may assume $(2n+1)a \ge 4\alpha -2n$. We specialize $Z_a\cup B$ to $Z_{a-f_1-e_1}\cup W\cup W' \cup B$
with $W$ a general union of $f_1-e_1$ schemes of type $(2n-1,2)$ with respect to $D$ and $e_1$ schemes of type $(2,2n-1)$ with respect to $D$.
By $\pounds$ and Lemma \ref{a2} it is sufficient to prove that $Z_{a-f_1}\cup \Res _D(W) \cup \Res _D(W')$ imposes independent conditions to $V[2]$.
Then we continue as in the proof of Lemma \ref{y3}.\end{proof}

We will prove simultaneously Theorem \ref{i2} and the next lemma.

\begin{lemma}\label{y5}
Fix an integer $t\ge 3$. Let $B\subset D$ be a general union of $e_1$ tangential schemes of $D$. For any $a\in \NN$ $Z_a\cup B$ gives the expected number of conditions to $V[t]$.
\end{lemma}

\begin{proof}[Proof of Theorem \ref{i2} and Lemma \ref{y5}:]
The theorem (resp. Lemma \ref{y5}) is true for $t=3$ by Lemma \ref{y3} (resp. Lemma \ref{y4}). Thus we may assume $t\ge 4$
and that both results holds for $V[x]$ for all $x$ such that $3\le x\le t-1$. Fix $a\in \NN$.

Take first the set-up of Theorem \ref{i2}. We specialize
$Z_a$ to a general $M:= Z_{a-f_1-e_1}\cup W\cup W'$ with $W$ a general union of $f_1$ schemes of type $(2n-1,2)$ with respect to
$D$ and $W'$ a general union of $e_1$ schemes of type $(2,2n-1)$ with respect to $D$. We have $\deg (M\cap D)=\alpha$. By $\pounds$ and Lemma \ref{a2} we have $V[0](-M\cap D)=0$. Thus it is sufficient to prove
that $Z_{a-f_1-e_1}\cup \Res _D(W)\cup \Res _D(W')$ imposes $\min \{t\alpha, \deg (Z_{a-f_1-e_1}\cup \Res _D(W)\cup \Res
_D(W'))\}$
conditions to $V[t-1]$. Note that $(t+1)\alpha -\deg (Z_a) =t\alpha -\deg (Z_{a-f_1-e_1}\cup \Res _D(W)\cup \Res _D(W'))$. Since
$\Res _D(W)$ is a general union of $f_1$ tangent vectors of $D$ and $\deg (D\cap  \Res _D(W)\cup \Res _D(W')) \le \alpha$, by Lemma \ref{a3} it is sufficient to prove that
$Z_{a-f_1-e_1}\cup \Res _D(W')$ gives the expected number of conditions to $V[t-1]$ (true by the inductive assumption for
Lemma \ref{y5}) and that $\dim V[t-2](-Z_{a-f_1-e_1}) \le \max \{0,t\alpha -(2n+1)(a-f_1-e_1) -(2n-1)e_1 -2f_1\}$.
If $t\ge 5$ we use the inductive assumption of the theorem
and that $\alpha \ge (2n-1)e_1 +2f_1$, i.e. $f_1\ge e_1$ (true because $e_1\le 3n-3$ and $\alpha \ge (2n+1)(4n-3)$). Now assume $t=4$.  We have $\deg (Z_a) -\deg (Z_{a-f_1-e_1}) = \alpha +(2n-1)e_1+2f_1$. If $(2n+1)a \ge 5\alpha = (10n-5)f_1 + 10e_1$ we need to prove $V[2](-Z_{a-f_1-e_1}) =0$ and so it is sufficient to have
$(2n+1)(a-f_1-e_1) \ge 3\alpha +2n-1-\mu = (6n-3)f_1+6e_1+2n-1-\mu$. Hence it is sufficient to have $(2n-3)f_1\ge (2n-5)e_1 +2n-1-\mu$. Since $2f_1\ge 2n$, it is sufficient to have $f_1\ge e_1$, which (since $e_1\le 2n-2$) it is true if $\alpha \ge (2n+1)(2n-2)$.

Now  we handle Lemma \ref{y5}. We specialize $B\cup Z_a$ to $M':= B\cup Z_{a-f_1+e_1}\cup W_1 \cup W_2$ with $W_1$
a general union of $f_1-e_1$ schemes of type $(2n-1,2)$ with respect to $D$ and $W'$ a general union of  $e_1$ schemes of type
$(2,2n-1)$ with respect to $D$. Taking $\Res _D$ we land in the case done for the proof of the theorem, except that $a$ may be
different and we only have $f_1-e_1$ tangent vectors. For $t=4$ to handle $V[2]$ we need the following computations. 

First assume $(2n+1)a +(2n-1)e_1 \ge 5\alpha$, i.e. $(2n+1)a \ge (10n-5)f_1+ (11-2n)e_1$. We we have to prove that
$V[2](Z_{a-f_1+e_1}) =0$. By Lemma \ref{y2} it is sufficient to prove that $(2n+1)(a-f_1+e_1) \ge   (6n-3)f_1 + 6e_1 +2n+2$.
Thus it is sufficient to have $(10n-5)f_1+(11-2n)e_1 \ge (8n-4)f_1 +(5-2n)e_1+2n+2$, i.e. $(2n-1)f_1 +6e_1\ge 2n+2$, which is
obviously true. If $(2n+1)a < 5\alpha$, it is sufficient to do the case with $(2n+1)a \ge 5\alpha -2n$. Use that
$(2n-1)f_1 +6e_1 \ge 4n+2$.\end{proof}

\begin{proof}[Proof of Corollary \ref{i2.0}:] A stronger statement is true for $n=2$ by \cite{co}. The case $n>2$ follows by induction on
$n$ by Theorem \ref{i2} applied to $Y:= (\PP^1)^{n-1}$, $V:= H^0(\Oo _{(\PP^1)^{n-1}}(d_1,\dots ,d_{n-1}))$.
\end{proof}

\begin{proof}[Proof of Corollary \ref{i2.1}:]
By \cite{av} and Theorem \ref{i2} it is sufficient to check when $(t+1)\binom{m+d}{m} \ge 4(m+1)^2+2(m+1)-4 = 4m^2+10m+2$.
\end{proof}

\section{$\PP^m\times \PP^1$, $\Oo (x,3)$, $x\le 3$, $m\le 4$}
In this section we prove the following proposition.

\begin{proposition}\label{x1}
The tangential variety of $\PP^m\times \PP^1$, $m\le 4$, with respect to $|\Oo _{\PP^m\times \PP^1}(x,3)|$ is not defective for $x =2,3$. 
\end{proposition}

We fix $o\in \PP^1$ and set $D:= \PP^m\times \{o\}$. For any positive integer $x$ set $Z_x:= Z_{0,x}$.

We warm up with the following easy result.

\begin{proposition}\label{c4}
Fix an integer $m>0$ and set $X= \PP^m\times \PP^1$. We have $h^1(\Ii _{Z_1}(1,3)) =h^0(\Ii _{Z_2}(1,3)) =0$.
\end{proposition}

\begin{proof}
We use induction on $m$, the case $m=1$ being true by \cite{co}. Take $m>1$ and set $M:= \PP^{m-1}\times \PP^1\in |\Oo _X(1,0)|$.  We specialize $Z_2$ (resp. $Z_1$) to a general union $A$
of two (resp. one) schemes of type $(2n-1,2)$ with respect to $M$. The inductive assumption gives $h^0(M,\Ii _{A\cap M}(1,3)) =0$ (resp. $h^1(\Ii _{A\cap M}(1,3))
=0$). Since $E:= \Res _M(A)$ is a general union of two (resp. one) tangent vectors of $M$,  $h^i(\Ii _E(0,1)) =0$,
$i=0,1$ (resp. $h^1(\Ii _E(0,1)) =0$).
\end{proof}

\begin{lemma}\label{c+}
Set $Y:= \PP^1$ and $X:= \PP^1\times \PP^1$. Let $A\subset X$ be a general union of $Z_2$ and one 2-point. Then
$h^1(\Ii _A(3,2)) =0$.
\end{lemma}

\begin{proof}
We specialize $A$ to a general union $B$ of a scheme of type $(2,2,1)$ with respect to $D = Y\times \{o\}$
and a scheme of type $(2,1)$. We have $h^i(D,\Ii _{D\cap B}(3,x)) =0$, $i=0,1$, $x\in \NN$, and $\Res _D(B)$ is a general
union of one 2-point of $X$ with support on $D$ and a point of $D$. Since $\Res _D(\Res _D(B))$ is a point and $h^1(D,\Ii
_{D\cap \Res _D(B)}(3,x)) =0$, $x\in \NN$, we conclude.
\end{proof}

\begin{lemma}\label{c1}
Set $Y:= \PP^1$ and $X:= \PP^2\times \PP^1$. Let $A\subset X$ be a general union of one tangential scheme and one $2$-point. Then
$h^1(\Ii _A(1,3)) =0$.
\end{lemma}

\begin{proof}
We specialize $A$ to a general union $B$ of a scheme of type $(2,2,1)$ with respect to $D:= Y\times \{o\}$
and a scheme of type $(2,1)$. We have $h^i(D,\Ii _{D\cap B}(3,x)) =0$, $i=0,1$, $x\in \NN$, and $\Res _D(B)$ is a general
union of one $2$-point of $X$ with support on $D$ and a point of $D$. Since $\Res _D(\Res _D(B))$ is a point and $h^1(D,\Ii
_{D\cap \Res _D(B)}(3,x)) =0$, $x\in \NN$, we conclude.
\end{proof}

\begin{lemma}\label{c++01}
Take $X=\PP^2\times \PP^1$, fix a line $H\subset \PP^2$ and set $M: =H\times \PP^1$. Let $A\subset X$ be a general 2-point with $A_{\red}\in M$. Then $h^0(\Ii _{Z_1\cup A}(1,3)) =1$.
\end{lemma}

\begin{proof}
We specialize $Z_1$ to a scheme $W$ of type $(5,2)$ with respect to $M$. Since $\Res _M(W\cup A)$ is a general union of a
tangent vector of $M$ and a point of $M$,  $h^0(\Ii _{\Res _M(W)}(0,3)) =1$. Thus it is sufficient to prove that
$h^0(M,\Ii _{M\cap (W\cup A)}(1,3))\le 1$. Take $L\in |\Oo _M(0,1)|$. We specialize $A\cap M$ to a 2-point $A'$ of $L$ whose
reduction is a general point of $M$. Taking $\Res _L$ we get $h^0(M,\Ii _{M\cap (W\cup A)}(1,3))=h^0(M,\Ii _{(M\cap W)\cup
\{p\}}(1,2)) =0$ (\cite{co} applied to the tangential scheme $M\cap W$ and Lemma \ref{a3}).
\end{proof}

\begin{lemma}\label{c0}
Set $Y:= \PP^1$ and $X:= \PP^1\times \PP^1$. Let $A\subset X$ be a general union of $Z_2$ and one 2-point. Then
$h^1(\Ii _A(3,2)) =0$.
\end{lemma}

\begin{proof}
We specialize $A$ to a general union $B$ of a scheme of type $(2,2,1)$ with respect to $D:= Y\times \{o\}$
and a scheme of type $(2,1)$. We have $h^i(D,\Ii _{D\cap B}(3,x)) =0$, $i=0,1$, $x\in \NN$, and $\Res _D(B)$ is a general
union of a $2$-point of $X$ with support on $D$ and a point of $D$. Since $\Res _D(\Res _D(B))$ is a point and $h^1(D,\Ii
_{D\cap \Res _D(B)}(3,x)) =0$, $x\in \NN$, we conclude.
\end{proof}

\begin{lemma}\label{c00}
Set $X:= \PP^1\times \PP^1$. Let $A\subset X$ be a general union of one tangential point and one 2-point. Then $h^0(\Ii _A(2,3)) =4$.
\end{lemma}

\begin{proof}
By \cite{co} we have $h^1(\Ii _{Z_2}(2,3)) =0$. Since $Z_2\supset A$, we get $h^1(\Ii _A(2,3)) =0$.
\end{proof}

\begin{lemma}\label{c01}
Set $X:= \PP^2\times \PP^1$ and $M=\PP^1\times \PP^1$. Let $B\subset X$ be a general  2-point of $X$ with $B_{\red}\in M$. Then $h^1(\Ii _{Z_1\cup B}(1,3)) =0$.
\end{lemma}

\begin{proof}
We specialize $Z_1$ to a scheme $A$ of type $(5,2)$ with respect to $M$.
Since $\Res _M(A)$ is a general tangent vector of $M$, its image by the projection onto the second factor of $X$ is a
tangent vector of $\PP^1$. Thus $h^0(\Ii _{\Res _M(A\cup B)}(0,3)) =2$. Thus it is sufficient to show that $h^i(M,\Ii _{A\cap
M}(1,3)) =0$, $i=0,1$. Fix $o\in \PP^1$ and set $D:= \PP^1\times \{o\}$. We specialize $A\cap M$ to a union
$G$ of a $2$-points whose reduction is in $D$ and a general tangential point $F$ of $M$. The scheme $\Res _D(G)$ is the union
of $F$ and a general point of $D$. Thus $\Res _D(G)$ gives independent conditions for $H^0(\Oo _M(1,2))$, because the
tangential point gives $5$ independent conditions to $H^0(\Oo _M(1,2))$ (\cite{co}) and $4$ conditions to $H^0(\Oo _M(1,1))$.
\end{proof}

\begin{lemma}\label{c2}
Set $X:= \PP^2\times \PP^1$ and $M:= \PP^1\times \PP^1$. Let $A\subset X$ be a general union of $Z_2$ and the scheme $E$ of type $(3,3)$ with respect to $M$
with both $E\cap M$ and $\Res _M(E)$ a 2-point of $M$. Then
$h^0(\Ii _A(2,3)) \le 6$.
\end{lemma}

\begin{proof}
We have $h^0(\Oo _X(2,3)) =24$ and $h^0(\Oo _M(2,3)) =12$. We specialize $A$ to $Z_1\cup W\cup E$ with $W$ a scheme of type $(5,2)$ with respect to $M$.
Since $h^0(M,\Ii _{W\cap M}(2,3)) =4$ (Lemma \ref{c00}), it is sufficient to prove that $h^0(\Ii _{Z_1\cup \Res _M(W)\cup \Res _M(E)}(1,3)) \le 2$. Since $\Res _M(W)$ is a general tangent vector of $M$ and $\deg (D\cap (\Res _M(W)\cup \Res _M(E)) \le 8$ , by Lemma \ref{a3} it is sufficient to prove that $h^0(\Ii _{Z_1}(0,3)) \le
2$. We have $h^0(\Ii _{Z_1}(0,3)) =1$, because a tangential scheme of $\PP^1$ imposes $3$ independent conditions to $H^0(\Oo
_{\PP^1}(3))$. We have $h^0(\Ii _{Z_1 \cup \Res _M(W)}(1,3)) \le 4$ by Lemma \ref{c01}.
\end{proof}

\begin{lemma}\label{c3.1}
Set $X:= \PP^2\times \PP^1$. Then $h^0(\Ii _{Z_2}(1,3)) =0$.
\end{lemma}

\begin{proof}
We have $h^1(\Ii _{Z_1}(1,3)) =0$, i.e. $h^0(\Ii _{Z_1}(1,3)) =2$. Use that a general tangent vector is contained in a general tangential scheme.
\end{proof}

\begin{lemma}\label{c3.2}
Set $X:= \PP^2\times \PP^1$ and $M:= \PP^1\times \PP^1\in |\Oo _X(1,0)|$. Let $A\subset M$ be a general $2$-point of $M$. Then
$h^1(\Ii _{Z_2\cup A}(2,3)) =0$.
\end{lemma}

\begin{proof}
We specialize $Z_2$ to $Z_1\cup B$ with $B$ a scheme of type $(5,2)$ with respect to $M$. We have $h^1(M,\Ii _{(B\cap M)\cup
A}(2,3)) =0$, because by \cite{co} two general tangential schemes of $M$ give $10$ independent conditions to $H^0(\Oo
_M(2,3))$. Thus it is sufficient to prove that $h^1(\Ii _{Z_1\cup E}(1,3))=0$, where
$E= \Res _M(B)$ is a general tangent vector of $M$, i.e. it is sufficient to prove that $h^0(\Ii _{Z_1\cup E}(1,3)) = 3$. Since $\deg (E) \le 6$, by Lemma \ref{a3} it is sufficient to observe
that $h^1(\Ii _{Z_1}(1,3)) =0$ and that $h^0(\Ii _{Z_1}(0,3)) \le 3$ (we have $h^0(\Ii _{Z_1}(0,3)) =1$). 
\end{proof}

\begin{lemma}\label{c3}
Set $X:= \PP^2\times \PP^1$. Let $B\subset X$ be a general union of $5$ tangential schemes and one 2-scheme. Then $h^1(\Ii _B(3,3)) =0$ and $h^0(\Ii _B(3,3)) =1$.
\end{lemma}

\begin{proof}
Since $\deg (B) = 39 = h^0(\Oo _X(3,3)) -1$, the two assertions are equivalent. Set $M:= \PP^1\times \PP^1$ seen as an element
of $|\Oo _X(1,0)|$. We have $h^0(\Oo _M(3,3))=16$. We specialize $B$ to a scheme $Z_2\cup A$ with $A$ general
union of $3$ schemes of type $(5,2)$ with respect to $M$ and a scheme type $(1,3)$ with respect to $M$. Since $M\cap A$ is a
general union of $3$ tangential schemes of $M$ and one point, by \cite{co} we have $h^i(M,\Ii _{M\cap A}(3,3)) =0$, $i=0,1$.
Thus it is sufficient to prove that $\Res _M(A)\cup Z_2$ imposes independent conditions to $H^0(\Oo _X(2,3))$. The scheme $\Res
_M(A)$ is a general union of one 2-point of $M$, $A$, and $3$ general tangent vectors of $M$. By Lemma \ref{a3} it is
sufficient to prove that $h^1(\Ii _{Z_2\cup A}(2,3)) =0$ (true by Lemma \ref{c3.2}) and that $h^0(\Ii _{Z_2}(1,3)) \le h^0(\Oo
_M(2,3))-\deg (Z_2\cup A\cup E)) = 1$ (true by Lemma \ref{c3.1}). 
\end{proof}

\begin{lemma}\label{c5}
Set $X:= \PP^3\times \PP^1$ and $M:= \PP^2\times \PP^1\in |\Oo _X(1,0)|$. Let $A\subset X$ be a general $2$-point of $X$ with $A_{\red}\in M$ and $B$ a scheme of type $(4,4)$ with respect to $M$ such that $B\cap M = \Res _M(B)$ is
one 2-point of $M$. Then $h^0(\Ii _{Z_2\cup A\cup B}(2,3)) \le 10$.
\end{lemma}

\begin{proof}
Take $N\in |\Oo _M(1,0)|$, $N\cong \PP^1\times \PP^1$.

\quad \emph{Claim 1:} Let $F\subset N$ be a general union of one tangential scheme and two $2$-points. Then $h^1(N,\Ii
_F(2,3))=0$.

\quad \emph{Proof of Claim 1:} Fix $L\in |\Oo _N(1,0)|$. We specialize $F$ to a general union $G$ of a tangential scheme and
two $2$-points with reduction contained in $L$. We have $h^i(L,\Ii _{G\cap L}(2,3)) =0$, $i=0,1$. We have $h^1(N,\Ii _{\Res
_L(G)}(1,3)) =0$ by \cite{co} and Lemma \ref{a3}.

We specialize $Z_2$ to a general union $U$ of two schemes of type
$(7,2)$ with respect to
$M$. Thus
$M\cap (U\cup A\cup B)$ is a general union of two tangential schemes of $M$, one $2$-point of $M$ and one point of $M$.

\quad \emph{Claim 2:} $h^1(M,\Ii _{M\cap (U\cup A\cup B)}(2,3)) =0$.

\quad \emph{Proof of Claim 2:} We specialize $U\cup A\cup B$ to the union $U'$ of a tangential scheme, a scheme of type
$(5,2)$ with respect to $N$ and two $2$-points with reduction contained in $N$. By Claim 1 we have $h^1(N,\Ii _{U'\cap
N}(2,3))=0$. The scheme $\Res _N(U')$ is a general union of one tangential scheme, a general tangent vector of $N$ and two
points of $N$. By Lemma \ref{a3} we have $h^1(\Ii _{\Res _N(U')}(1,3))=0$, proving Claim 2.

Since $\Res _M(U\cup A\cup B)$ is a general union of
$2$ tangent vectors of $M$, one point of $M$ and one $2$-point of $M$ and $h^0(\Oo _X(1,3))-h^0(\Oo _X(0,3)) =h^0(\Oo _M(1,3))=
12$, we get the lemma.
\end{proof}

\begin{lemma}\label{c6}
Take $Y:=\PP^4$, $X:= Y\times \PP^1$, $D:= \PP^4\times \{o\}$ and $V:= H^0(\Oo _{\PP^4}(3))$. Then $\dim V[3](-Z_{12}) =
\dim V[3] -\deg (Z_{12})$ and $V[3](-Z_{13})=0$.
\end{lemma}

\begin{proof}
We recall (\cite{bcgi, cgg}) that $\dim V = 35$, $3$ general tangential points of $Y$ give independent conditions to $V$, but
$4$ not. We specialize $Z_x$, $x\in \{12,13\}$, to the union of $Z_{x-6}$ and the union $A$ of $3$ schemes of type $(9,2)$ with
respect to
$D$ and $3$ schemes of type $(2,9)$ with respect to $D$. Since $h^i(D,\Ii _{D\cap A}(3)) =0$, $i=0,1$, by Lemma \ref{a2}
it is sufficient to prove that $Z_6\cup \Res _D(A)$ gives independent conditions to $V[2]$ and that $V[2](-(Z_7\cup \Res
_D(A)) =0$. Let $B\subset D$ be a general union of $3$ tangential schemes of $D$. By Lemma \ref{a3} it is sufficient to
prove that
$Z_6\cup B$ gives independent conditions to
$V[2]$, that
$\dim V[1](-Z_6) \le 6$, that $\dim V[2](-(Z_7\cup B)) \le 6$ and that $V[1](-Z_7) =0$. To control $V[2]$ we specialize
$Z_{x-6}\cup B$ to $Z_{x-9}\cup B\cup E$ with $E$ a general union of $3$ schemes of type $(9,2)$, use that $h^i(\Ii _{(D\cap
E)\cup B}(3)) =0$, $i=0,1$, and then use again Lemma 3. To prove that $V[1](-Z_6) = 0$ we specialize $Z_6$ to a general
union of $3$ schemes of type $(9,2)$ and $3$ schemes of type $(2,3)$.
\end{proof}

\begin{proof}[Proof of Proposition \ref{x1}:]
Since the case $m=1$ is true by \cite{co}, we may assume $m\ge 2$. Fix a hyperplane $H\subset \PP^m$ and set $M:= H\times \PP^1$ and $X:= \PP^m\times \PP^1$.
We have $h^0(\Oo _X(x,y)) =(y+1)\binom{m+x}{m}$ for all $(x,y)\in \NN^2$.
A tangential scheme of $X$ (resp. $M$) has degree $2m+3$ (resp. $2m+1$). We write $Z_x$ for a general union of $x$ tangential
schemes of $X$.

\quad (a) Take $m=2$.

\quad (a1) Consider $|\Oo _X(2,3)|$. We have $h^0(\Oo _X(2,3)) = 24$ and $h^0(\Oo _M(2,3)) = 12$. 
To prove that $h^1(\Ii _{Z_3}(2,3)) =0$ we specialize $Z_3$ to $Z_1\cup W$ with $W$ a general union of $2$ schemes of type $(5,2)$ with respect to $M$. We have $h^1(M,\Ii _{W\cap M}(2,3)) =0$ by \cite{co} and hence it is sufficient to prove that $h^1(\Ii _{Z_1\cup \Res _M(W)}(1,3))=0$.
Since $h^1(\Ii _{Z_1}(1,3))=0$, $h^0(\Ii _{Z_1}(0,3)) = 4-3$ and $\Res _M(W)$ is a general union of $2$ tangent vectors of $M$, it is sufficient to apply Lemma
\ref{a3}.

To get 
that $h^0(\Ii _{Z_4}(2,3)) =0$  we specialize $Z_4$ to $Z_1\cup W\cup W'$ with $W$ as above and $W'$ a scheme of type
$(3,3,1)$ with respect to $M$. It is sufficient to prove that $h^0(\Ii _{Z_1\cup \Res _M(W) \cup \Res _M(W')}(1,3)) =0$. By
Lemma \ref{a3} it is sufficient to prove that $h^0(\Ii _{Z_1\cup \Res _M(W')}(1,3))  \le 2$ (true by Lemma \ref{c++01}) and
$h^0(\Ii _{Z_1\cup \Res _M(\Res _M(W'))}(0,3))=0$ (true, because $\Res _M(\Res _M(W'))$ is a general point of $M$ and so $\deg
(\pi _2(Z_1\cup \Res _M(\Res _M(W')))) =4$).

\quad (a2) Consider $|\Oo _{\PP^2\times \PP^1}(3,3)|$. We have $h^0(\Oo _X(3,3)) = 40$ and $h^0(\Oo _M(3,3)) =16$. To prove that $h^1(\Ii _{Z_5}(3,3)) =0$ we specialize $Z_5$ to
$Z_2\cup A$ with $A$ a general union of $3$ schemes of type $(5,2)$ with respect to $M$. Since $h^1(M,\Ii _{A\cap M}(3,3))
=0$ (\cite{co}), it is sufficient to prove that $h^1(\Ii _{Z_2\cup E}(2,3)) =0$, where $E$ is a general union of $2$ tangent
vectors. By Lemma \ref{a3} it is sufficient to use that $h^1(\Ii _{Z_2}(2,3)) =0$ (step (a1))  and that $h^0(\Ii _{Z_2}(1,3))
=0$ (Proposition \ref{c4}). To prove that
$h^0(\Ii _{Z_6}(3,3)) =0$  we specialize $Z_6$ to  $Z_2\cup A\cup A'$ with $A$ a general union of
$3$ schemes of type $(5,2)$ with respect to $M$ and $A'$  a scheme of type $(1,3,3)$ with respect to $M$. We have $h^1(M,\Ii
_{A\cap M}(3,3)) =0$, $i=0,1$, by \cite{co}.
The scheme $\Res _M(A)$ is a general union of $3$ tangent vectors of $M$ and a scheme $E$ of type $(3,3)$ with respect to $M$
such that $E\cap M$ and $\Res _M(E)$ are $2$-points of $M$. We have $h^0(\Ii _{Z_2\cup \Res _M(E)}(1,3))  \le 2$ by Lemma \ref{c1}.
By Lemma \ref{a3} it is sufficient to prove that $h^0(\Ii _{Z_2\cup E}(2,3)) \le 6$, which is true by Lemma \ref{c2}.

\quad (b) Take $m=3$. Thus a tangential scheme has degree $9$.

\quad (b1) We have $h^0(\Oo_X(2,3)) = 40$ and $ h^0(\Oo _M(2,3)) = 24$. To prove that $h^1(\Ii _{Z_4}(2,3)) =0$ we specialize
$Z_4$ to $Z_1\cup W$ with $W$ a general union of $3$ schemes of type $(7,2)$ with respect to $M$. We have $h^1(M,\Ii _{W\cap
M}(2,3)) =0$ by step (a2) and hence it is sufficient to prove that $h^1(\Ii _{Z_1\cup \Res _M(W)}(1,3))=0$. By Lemma \ref{a3}
it is sufficient to observe that $h^1(\Ii _{Z_1}(1,3)) =0$ (Proposition \ref{c4}) and that $1 = h^0(\Ii _{Z_1}(0,3))
\le 5= \max \{0, h^0(\Oo _X(1,3)) -\deg (Z_1\cup \Res _M(W))\}$.

\quad (b2) We have $h^0(\Oo _X(3,3))  = 80$ and $h^0(\Oo _M(3,3)) = 40$. To prove that $h^0(\Ii _{Z_9}(3,3)) =0$ we specialize $Z_9$ to $Z_2\cup A$ with $A$
a general union of $5$ schemes of type $(7,2)$ with respect to $M$, one scheme of type $(4,4,1)$ with respect to $M$ and one
scheme of type $(1,4,4)$ with respect to $M$. We have $h^i(M,\Ii _{A\cap M}(3,3)) =0$, $i=0,1$, by Lemma \ref{c3}. Thus it is
sufficient to prove that
$h^0(\Ii _{Z_2\cup E\cup F\cup G}(2,3))=0$, where $E$ is a general union of $5$ tangent schemes of $M$, $F$ has type $(4,1)$ with respect to $M$  and $G$ has type $(4,4)$ with respect to $M$ (both $G\cap M$ and $\Res _M(G)$ are 2-points of $M$. By Lemma \ref{a3} it is sufficient to prove
that $h^0(\Ii _{Z_2 \cup \Res _M(F)\cup \Res _M(G)}(2,3)) \le 10$ (true by Lemma \ref{c5}) and that $h^0(\Ii _{Z_2}(1,3)) =0$ (true by the case $m=3$ of
Proposition \ref{c4}).

The proof that $h^1(\Ii _{Z_8}(3,3)) =0$, i.e. that $h^0(\Ii _X(3,3)) =8$ is easier; it may be done with minimal modifications to the proof that $h^0(\Ii _{Z_9}(3,3)) =0$, but
we prefer to show that it  follows  from the statement that $Z_9$ imposes $9^2 -1$ conditions to
$H^0(\Oo _X(3,3))$. Indeed, it first implies that $h^0(\Ii _{Z_7}(3,3))= 4\cdot 20 -7\times 9$ and then that the addition of $2$
general $Z_1$ gives $17$ independent conditions to $H^0(\Ii _{Z_7}(3,3))$. Thus adding the first $Z_1$ must give $9$
independent conditions to $H^0(\Ii _{Z_7}(3,3))$.

\quad ({c}) Take $m=4$. Thus a tangential scheme has degree $11$. Fix a hyperplane $H'$ of $M$ and a hyperplane $H''$ of $H'$ and set $N:= H'\times \PP^1$
and $N':= H''\times \PP^1$. We have $h^0(\Oo _N(3,3)) =40$ and $h^0(\Oo _{N'}(3,3)) = 16$.
Let $O' \subset N'$ be a general union of $3$ tangential schemes of $N'$. By \cite{co} we have $h^1(N',\Ii _{O'}(3,3)) =0$.

\quad \emph{Claim 1:} Let $O \subset N$ be a general union of $5$ tangential schemes of $N$ and one 2-point of $N$. Then $h^1(N,\Ii _O(3,3)) =0$.

\quad \emph{Proof of Claim 1:} We specialize $O$ to $A_1\cup A_2$ with $A_1$ a general union of $2$ tangential schemes of $N$ and one 2-point of $N$
and $A_2$ is a general union of $3$ schemes of of type $(5,2)$ with respect to $N'$. We just proved that $h^1(N',\Ii _{A_2\cap N'}(3,3)) =0$. By Lemma \ref{a3}
it is sufficient to prove that $h^1(N,\Ii _{A_1}(2,3)) =0$ and that $h^0(N,\Ii _{A_1}(1,3)) =0$. We have $h^1(N,\Ii _{A_1}(2,3)) =0$ by step (a1),
because $A_1$ is contained in a general union of $3$ tangential schemes of $N$. We have $h^0(N,\Ii _{A_1}(1,3)) =0$ by Proposition \ref{c4}, because $A_1$ contains a general union of $2$ tangential schemes of $N$.

\quad (c1) We have $h^0(\Oo _X(2,3)) = 60$ and $h^0(\Oo _M(2,3)) =40$.  To prove that $h^1(\Ii _{Z_5}(2,3)) =0$ we specialize $Z_5$ to $Z_1\cup E$ with $E$ a general union
of $4$ schemes of type $(9,2)$ with respect to $M$. By step (b1) and Lemma \ref{a3} it is sufficient to observe that $h^1(\Ii _{Z_1}(1,3)) =0$ (Proposition \ref{c4}) and that
$1 = h^0(\Ii _{Z_1}(0,3)) \le  1= \max \{0,h^0(\Oo _X(1,3)) - \deg (Z_1)-\deg (\Res _M(E))\}$. 

\quad (c2) We have $h^0(\Oo _X(3,3)) = 140$ and $h^0(\Oo _M(3,3)) =80$.

\quad \emph{Claim 2:} Let $K \subset M$ be a general union of $8$ tangential schemes of $M$ and one 2-point of $M$. Then $h^1(M,\Ii _K(3,3)) =0$.

\quad \emph{Proof of Claim 2:}  We specialize $K$ to $B_1\cup B_2\cup B_3$ with $B_1$ a general union of $3$ tangential schemes
of $M$, $B_2$ a scheme of type $(4,1)$ with respect to $N$ and $B_3$ a general union of $5$ schemes of type $(7,5)$ with
respect to
$N$. Claim 1 gives
$h^1(N,\Ii _{B_2\cap N}(3,3)=0$. By Lemma \ref{a3} it is sufficient to prove that $h^1(M,\Ii _{B_1}(2,3))=0$ (true by step (a)) and that
$h^0(M,\Ii _{B_1}(1,3)) =0$ (true by Proposition \ref{c4}). To prove that $h^1(M,\Ii _{B_1}(2,3))=0$

To prove that $h^1(\Ii _{Z_{12}}(3,3)) =0$ we specialize $Z_{12}$ to $Z = Z_3\cup F\cup G$ with $F$
a general union of $8$ schemes of type $(9,2)$ with respect to $M$ and $G$ a scheme of type $(5,5,1)$ with respect to $M$. By Claim
2 we have
$h^1(M,\Ii _{Z\cap M}(3,3)) =0$. The scheme $\Res _M(Z)$ is a general union of $Z_3$ and $U:= \Res _M(G)$ (which is a 2-point
of $X$ whose reduction is a general $p\in M$). By Lemma \ref{a3} it is sufficient to prove $h^1(\Ii _{Z_3 \cup U}(2,3)) =0$
and $h^0(\Ii _{Z_3\cup \{p\}}(1,3)) =0$ (true by Proposition \ref{c4}). To prove that $h^1(\Ii _{Z_3 \cup U}(2,3))
=0$ we specialize $Z_3\cup U$; we are allowed to use step (c1), because $Z_3\cup U$ is projectively equivalent to a subscheme of a general
union of $4$ tangential schemes of $X$.
\end{proof}

\section{Joins of the varieties and its tangent developable}

For any $X$, $\Ll \in \mathrm{Pic}(X)$ and $V\subseteq H^0(\Ll)$ consider the following Assumption $\pounds \pounds$ which the
pair $(X,V)$ may have:

\quad \textbf{Assumption} $\pounds \pounds$: For a general union $Z$ of finitely many tangential schemes
of
$X$ and finitely many 2-points of $X$ the linear system $V(-Z)$ has the expected dimension $\max \{\alpha -\deg (Z), 0\}$.

Some of the statements of section \ref{Sy} may be improved (allowing a lower $\alpha$) if $V$ has $\pounds \pounds$.

In this section (except Lemmas \ref{oo2} and \ref{cd1}) we take $X = \PP^n$ for some $n$ and $V = H^0(\Oo _X(t))$. Let
$Z_{a,b}$ denote a general union of $a$ 2-points and $b$ tangential schemes of $X$.

\begin{proposition}\label{ux2}
Take $X=\PP^2$. We have $h^1(\Ii _{Z_{a,b}}(t))\cdot h^0(\Ii _{Z_{a,b}}(t))\ne 0$ if and only if
$(t,a,b)\in
\{(2,2,0), (4,5,0),(3,0,2)\}$, with $\delta:= h^1(\Ii _{Z_{a,b}}(t)) =h^0(\Ii _{Z_{a,b}}(t)) =1$ in each case and with $|\Ii
_{Z_{a,b}}(t)|$ either a multiple line or (case $(4,5,0)$) a multiple conic. 
\end{proposition}

\begin{proof}
We obviously have $h^0(\Ii _{Z_{a,b}}(1)) =0$ if $(a,b)\ne (0,0)$. The set $|\Ii _{Z_{2,0}}(2)|$ is the double line
containing $(Z_{2,0})_{\red}$  and in this case we have $\delta =1$. Obviously $h^0(\Ii _{Z_{a,b}}(2)) =0$ if either $a+b\ge 3$ or
$a+b=2$ and
$b>0$. We have
$h^1(\Ii _{Z_{0,1}}(2)) =0$ (\cite{cgg}). For $t\ge 3$ we know that the only exceptional cases $(t,a,b)$ with either $a=0$ or $b=0$
are
$(4,5,0)$ (with $\delta =1$ and $|\Ii _{Z_{5,0}}(4)| = 2C$ with $C$ the smooth conic containing $(Z_{5,0})_{\red}$) and
$(3,0,2)$ (with
$\delta =1$, the only element of
$|\Ii _{Z_{0,2}}(3)|$ being the triple line spanned by $(Z_{0,2})_{\red}$). We only check the cases $ab\ne 0$. Fix a line
$L\subset
\PP^2$.

\quad (a) Take $t=3$. We specialize $Z_{a,b}$ to $E:= Z_{a-1,b-1}\cup A\cup B$ with $A$ a scheme of type $(2,2,1)$ with respect
to
$H$
(i.e. $A$ is a general tangential scheme with $A_{\red}\in H$) and $B$ of type $(2,1)$ with respect to $H$. Since $h^i(H,\Ii
_{E\cap H}(3))=0$, $i=1,2$, it is sufficient to prove that $\Res _H(E)$ gives the expected number of conditions to $H^0(\Oo
_{\PP^2}(2))$. Since $\Res _H(B)$ is a general point of $H$, by Lemma \ref{a3} it is sufficient to prove that $h^0(\Ii
_{Z_{a-1,b-1}\cup \Res _H(\Res _L(A))}(1)) \le \max \{0,6-3(a-1) -5(b-1) -1\}$ (obvious) and
that $Z_{a-1,b-1}\cup \Res _H(A)$ gives the expected number of conditions to $H^0(\Oo _{\PP^2}(2))$. We conclude, unless $(a,b)
=(2,1)$. In this case we specialize $Z_{2,1}$ to $E':= Z_{0,1}\cup B'$ with $B'$ general union of two schemes of type
$(2,1)$
with respect to $H$. By Lemma \ref{a3} it is sufficient to observe that $h^1(\Ii _{Z_{0,1}}(2)) =h^0(\Ii _{Z_{0,1}}(1)) =0$.

\quad (b) Take $t=4$. We specialize $Z_{a,b}$ to $F:= Z_{a-1,b-1} \cup M_1\cup M_2$ with $M_1$ a scheme of type $(3,2)$ with
respect to $H$ and $M_2$ a scheme of type $(2,1)$ with respect to $H$. By Lemma \ref{a3} it is sufficient to prove that
$h^0(\Ii _{Z_{a-1,b-1}}(2)) \le \max \{0,10 -3(a-1) -5(b-1)-3\}$ (obvious even when $(a-1,b-1) =(2,0)$) and that $Z_{a-1,b-1}$
gives the expect number of conditions to $H^0(\Oo _{\PP^2}(3))$ (true by the case $t=3$, unless $(a-1,b-1) = (0,2)$).
Now assume $a=1$ and $b=3$. Since $h^0(\Ii _{Z_{0,3}}(4)) =1$ (\cite{cgg}), we have $h^0(\Ii _{Z_{1,3}}(4)) =0$

\quad ({c}) Assume $t\ge 5$ and that in the use of Lemma \ref{a3} we do not land in one of the exceptional cases with $t'=3,4$
(for which see step (d)). By Remark \ref{stu1} we may assume $3a+5b \ge \binom{t+2}{2} -2$. Write $3x+2y = t+1$ with $y\in \{0,1,2\}$. If $b\ge x$
and
$a\ge y$ (resp.
$b>x$ and
$a<y$, i.e.
$a=1$ and $y=2$) we specialize
$Z_{a,b}$ to
$F:= Z_{a-y,b-x}\cup A\cup B$ (resp. $F':= Z_{0,b-x-1}\cup A\cup B'\cup B''$) with $A$ a general union of $x$ schemes of type
$(3,2)$ with respect to $H$, $B$ a general union of $y$ schemes of type $(2,1)$, $B'$ a scheme of type $(2,1)$ and $B''$ a
scheme of type $(2,3)$ with respect to $H$. By Lemma \ref{a3} it is sufficient to prove
that $h^0(\Ii _{Z_{a-y,b-x}}(t-2)) \le \max \{0,\binom{t+1}{2} -3(a-y) -5(b-x) -2x -y\}$, which is true by the inductive
assumption
if $(t-2,a-y,b-x)$ is not an exceptional case (resp.  $h^0(\Ii _{Z_{0,b-x-1}}(t-2)) \le \max \{0,\binom{t+1}{2} -5(b-x-1)
- 2x-1-3\}$) and that $Z_{a-y,b-x}$ (resp. $Z_{0,b-x-1}\cup \Res _H(B')$) gives the expected number of conditions
to $H^0(\Oo _{\PP^2}(t-1))$. The first check is true, unless $(t-1,a-y,b-x)$ is in one of the exceptional cases. For the check
of the `` resp. '' part we use the following trick. If $h^0(\Ii _{Z_{0,b-x-1}}(t-1)) =0$, then we are done. If $h^0(\Ii
_{Z_{0,b-x-1}}(t-1)) \ne 0$, then by the inductive assumption (note that $(t-1,0,b) \ne (3,0,2)$ because $t\ge 5$) we get
$h^1(\Ii _{Z_{0,b-x-1}}(t-1)) =0$. Since $3+5b \ge \binom{t+2}{2}-2$, the inductive assumption gives $h^0(\Ii
_{Z_{0,b-x-1}}(t-2))  =0$ (note that $(t-2,0,b-x-1) \ne (3,0,3)$ if $t=5$, because in this case we have $y=0$). Thus the
residual exact sequence of $H$:$$0 \to \Ii _{Z_{0,b-x-1}}(t-2) \to \Ii _{Z_{0,b-x-1}\cup \Res _H(B')}(t-1) \to \Ii _{\Res
_H(B'),H}(t-1) \to 0$$shows that we are done if the degree $3$ scheme $\Res _H(B')$ gives the expected number of conditions
to some (non-complete) linear system on $H$. This is a very particular case of \cite{cm}, since the support of $\Res _H(B')$
is a general point of $H$ and $\dim H=1$.

Now assume $b<x$. Write $t+1 = 3b+2y_1+y_2$ with $y_1\in \NN$ and $y_2\in \{0,1\}$. Since $3a+5b \ge \binom{t+2}{2} -2$,
we have $a\ge y_1+y_2$. We
degenerate $Z_{a,b}$ to $G:= Z_{a-y_1-y_2}\cup M_3\cup M_4$ with $M_3$ a general union of $y_1$ schemes of type $(2,1)$ with
respect to $H$ and $M_4$ a general union of $y_2$ schemes of type $(1,2)$ with respect to $H$. By Lemma \ref{a3} we conclude by
the inductive assumptions unless $t=5,6$ and $Z_{0,b-y_1-y_2}$ is an exceptional case for $H^0(\Oo _{\PP^2}(4))$.

\quad (d) Now assume $t=5,6$ and that in step ({c}) we landed in an exceptional case for $\Oo _{\PP^2}(t-1)$ or
$\Oo _{\PP^2}(t-2)$. If we landed in an exceptional case for $\Oo _{\PP^2}(t-1)$, then $t=5$ and we were considering
some $Z_{a',b'}$ with $a' =5$ and $b'=0$. For $t=5$ we have $(x,y) =(3,0)$. Thus $1\le b\le 2$. First assume
$b=2$. We get $a'= a$ and so $a=5$, $3\cdot 5 + 5\cdot 2 > 23 = \binom{7}{2} +2$, contradicting our assumption. If $b=1$ we
specialize $Z_{5,1}$ to $G' = Z_{2,1} \cup G''$ with $G''$ a general union of $3$ schemes of type $(2,1)$ with respect to $H$
and conclude by Lemma \ref{a3}.

Now assume that we landed in an exceptional case for $\Oo _{\PP^2}(t-2)$ (and hence $5\le t\le 6$). In each of the exceptional
cases the defect,
$\delta$, is $1$. Thus we only need to observe that in each of these cases we applied Lemma \ref{a3} to a scheme $W$
with $\deg (W\cap H) <t$.
\end{proof}

\begin{lemma}\label{oo2}
Take $X= \PP^1\times \PP^1$. We have $h^0(\Ii _{Z_{7,1}}(4,4)) =0$.
\end{lemma}

\begin{proof}
Fix $D\in |\Oo _X(0,1)|$. We specialize $Z_{7,1}$ to $Z_{6,0}\cup A\cup B$ with $A$ a scheme of type $(3,2)$ with respect to $D$ and $B$ a scheme of type $(2,1)$ with respect to $Q$. By Lemma \ref{a3} it is sufficient to prove that $h^1(\Ii _{Z_{6,0}}(4,3)) =0$ (true by \cite{laf}) and $h^0(\Ii _{Z_{6,0}}(4,2)) =0$ (true because \cite{laf} implies $h^0(\Ii _{Z_{5,0}}(4,2))
=1$).
\end{proof}

\begin{lemma}\label{cd1}
Take $X= \PP^1\times \PP^1$. We have $h^i(\Ii _{Z_{5,2}}(4,4)) =0$, $i=0,1$.
\end{lemma}

\begin{proof}
Since $\deg (Z_{5,2}) =25$, we have $h^0(\Ii _{Z_{5,2}}(4,4)) =h^1(\Ii _{Z_{5,2}}(4,4))$. Fix $D\in |\Oo _X(0,1)|$. We
specialize
$Z_{5,2}$ to $Z_{4,1}\cup A\cup B$ with $A$ of type $(3,2)$ with respect to $D$ and $B$ of type $(2,1)$ with respect to $D$.
By Lemma \ref{a3} it is sufficient to prove that $h^1(\Ii _{Z_{4,1}}(4,3)) =0$ and that $h^0(\Ii _{Z_{4,1}}(4,2)) \le \max
\{0,h^0(\Oo _X(4,3))-\deg (\Res _D(Z_{5,2}))\}=1$ (obvious because $h^0(\Ii _{Z_{5,0}}(4,1)) =1$ by \cite{laf}). To prove that
$h^1(\Ii _{Z_{4,1}}(4,3)) =0$ we specialize $Z_{4,1}$ to $Z_{3,0}\cup A\cup B$ with $A, B$ as above and use that $h^1(\Ii
_{Z_{3,0}}(4,2)) =0$ and $h^0(\Ii
_{Z_{3,0}}(4,1)) =1$ (\cite{laf}).
\end{proof}

\begin{proposition}\label{ux3}
Take $X:= \PP^3$. If $t\ge 3$ for any $(a,b)\in \NN^2$ $Z_{a,b}$ gives the expected number of conditions to $H^0(\Oo _X(t))$,
except in the following cases:
\begin{enumerate}
\item $t=3$, $a=0$, $b=3$; we have $|\Ii _{Z_{0,3}}(3)| =3M$, where $M$ is the plane spanned by $(Z_{0,3})_{\red}$ and
$h^1(\Ii _{Z_{0,3}}(3)) =2$;
\item $t=4$, $a=9$, $b=0$; we have  $|\Ii _{Z_{0,3}}(3)| =2Q$, where $Q$ is the smooth quadric containing $(Z_{9,0})_{\red}$
and $h^1(\Ii _{Z_{0,3}}(3)) =2$; 
\item  $t=4$, $a=7$, $b=1$; we have  $|\Ii _{Z_{0,3}}(3)| =2Q$, where $Q$ is the smooth quadric containing the
$7$ points $(Z_{7,0})_{\red}$ and the tangent vector used to define $Z_{0,1}$; hence $h^1(\Ii _{Z_{0,3}}(3)) =1$.
\end{enumerate}
\end{proposition}

\begin{proof}
The exceptional cases with either $a=0$ or $b=0$ are known by the Alexander-Hirschowitz theorem and \cite{bcgi, cgg}. Outside
these exceptional cases we only need to test $(a,b)$ with $a>0$, $b>0$ and $|4a+7b-\binom{t+3}{3}| \le 3$. We first check that
if $t=4$ we have $h^0(\Ii _{Z_{7,1}}(4)) =1$ with $|\Ii _{Z_{7,1}}(4)| = \{2Q\}$, where $Q$ is the only quadric containing the $7$
points appearing with multiplicity $1$ in $Z_{7,1}$ and the tangent vector $v$ defining the tangential scheme defining the
tangential scheme of $Z_{7,1}$. We have $2Q\in |\Ii _{Z_{7,1}}(4)|$ and $Q$ is smooth, because the union of these $7$ points
and the tangent vector is a general of $7$ points and a general tangent vector. Note that $Z_{7,1}\cap Q$ is a general union of
a tangential scheme of
$Q$ and
$7$ 2-points of
$Q$. By Lemma \ref{oo2} we have $h^0(Q,\Ii _{Z_{7,1}\cap Q}(4)) =0$. The scheme $\Res _Q(Z_{7,1})$ is a general union of $7$
points and one tangent vector
and hence it is contained in a unique quadric. Since it is also contained in $Q$, we get $|\Ii _{Z_{7,1}}(4)| = 2Q$ and hence
$h^1(\Ii _{Z_{7,1}}(4))=1$.

\quad (a) Assume $t=3$ and $(a,b)\ne (0,3)$. Since $h^0(\Ii _{Z_{0,4}}(3)) =h^0(\Ii _{Z_{1,3}}(3)) =0$
(the only element of $|\Ii _{Z_{0,3}}(3)|$ is a triple plane), we may assume $1\le b \le
2$. 

First assume $b=2$. To get $h^1(\Ii _{Z_{1,2}}(3)) =0$ we specialize $Z_{1,3}$ to $E:= Z_{0,1}\cup W\cup W'$ with $W$ of
type $(5,2)$ with respect to $H$ and $W'$ of type $(3,1)$ with respect to $H$. By Proposition \ref{ux2} we have $h^1(H,\Ii
_{H\cap (W\cup W')}(3)) =0$. By Lemma \ref{a3} it is sufficient to prove that $h^1(\Ii _{0,1}(2))=0$ (true by \cite{cgg}) and
that $h^0(\Ii _{Z_{0,1}}(1)) =0$ (obvious). To prove that $h^0(\Ii _{Z_{2,2}}(3)) =0$ we specialize $Z_{2,2}$ to $Z_{0,1}\cup
W\cup W'\cup W''$ with $W''$ of type $(3,1)$ with respect to $H$. We first use that $h^0(H,\Ii _{H\cap (W\cup W'\cup
W'')}(3))=0$
(Proposition \ref{ux2}) and then use Lemma \ref{a3} as in the case $(a,b)=(1,2)$ just done.

Now assume $b=1$. To prove that $h^1(\Ii _{Z_{3,1}}(3)) =0$ (i.e. $h^0(\Ii _{Z_{3,1}}(3)) =1$ and hence $h^0(\Ii
_{Z_{4,1}}(3)) =0$) we specialize $Z_{3,1}$ to $Z_{0,1}\cup M_1$ with $M_1$ a general union of $3$ schemes of type $(3,1)$
with respect to $H$ and then use the Alexander-Hirschowitz theorem in $H$ and Lemma \ref{a3}.

\quad (b) Assume $t=4$ and $(a,b)\notin \{(9,0),(7,1)\}$. 

First assume $b\ge 3$. We specialize $Z_{a,b}$ to $Z_{a,b-3}\cup W$
with $W$ a general union of $3$ schemes of type $(5,2)$ with respect to $H$. We first apply Proposition \ref{ux2} and then
Lemma \ref{a3}.

Now assume $b=2$. To prove that $h^1(\Ii _{Z_{5,2}}(4)) =0$ it is sufficient to prove
that $|\Ii _{Z_{5,2}}(4)| =2Q$, where $Q$ is the only quadric surface (it is a smooth quadric surface) containing the $5$
points of $(Z_{5,0})_{\red}$ and the two tangent vectors used to define the two tangential schemes of $Z_{5,2}$. To prove that
$|\Ii _{Z_{5,2}}(4)| =2Q$ it is sufficient to prove that $h^0(Q,\Ii _{Z_{5,2}\cap Q}(4)) =0$ (true by Lemma \ref{cd1}, because
$Q\cong
\PP^1\times \PP^1$ and $Z_{5,2}\cap Q$ is a general union of $2$ tangential schemes of $Q$ and $5$ 2-points of $Q$).
Since $h^0(\Ii _{Z_{5,2}}(4))=1$, we have $h^0(\Ii _{Z_{a,2}}(4)) =0$ for all $a\ge 6$.

Now assume $b=1$. Since $h^0(\Ii _{Z_{7,1}}(4)) = h^1(\Ii _{Z_{7,1}}(4))=1$, obviously  $h^0(\Ii _{Z_{a,1}}(4)) =0$ for
all $a\ge 8$, but it is also easy to check that $h^1(\Ii _{Z_{6,1}}(4)) =0$ in the following way. Let $x$ be the
maximal integer
$\le 6$ such that $h^1(\Ii _{Z_{x,1}(}4)) =0$. Note that $h^1(\Ii _{Z_{7,1}}(4))\ge 7-x$.

\quad ({c}) Assume $t\ge 5$. We work by induction on $t$ (as in the proof of Proposition \ref{ux2}), because for $t=5$ taking
$\Res _H$ we never land in an exceptional case.
\end{proof}

\begin{proof}[Proof of Theorem \ref{uu1}:]
Propositions \ref{ux2} and \ref{ux3} cover the cases $n=2,3$.

\quad (a) Assume $n=4$. We do here the cases $t=4,5,6$, because the inductive proof done in steps ({c}) and (d) works verbatim. We have
$h^0(\Oo _{\PP^3}(3)) = 20$, $h^0(\Oo _{\PP^3}(4)) =35$, $h^0(\Oo _{\PP^3}(5)) =56$, $h^0(\Oo _{\PP^3}(6))=84$, $h^0(\Oo
_{\PP^4}(2)) =15$,
$h^0(\Oo _{\PP^4}(3)) =35$,
$h^0(\Oo _{\PP^4}(4)) =70$,
$h^0(\Oo _{\PP^4}(5)) =126$ and $h^0(\Oo _{\PP^4}(6)) =210$. By assumption either $a\ge 7$ or $b\ge 4$.

\quad (a1) Assume $t=4$, $b\ge 4$ and $a\ge 3$. We specialize $Z_{a,b}$ to $Z_{a-1,b-4}\cup W\cup M\cup N$ with $W$ a
general union of
$3$ schemes of type $(7,2)$ with respect to $H$, $M$ is a scheme of type $(2,7)$ with respect to $H$ and $N$ is a scheme of
type
$(4,1)$ with respect to $H$. In $H$ we apply Proposition \ref{ux3} with $(t,a,b) =(4,3,3)$ and Lemma \ref{a2}. Then we use
Lemma \ref{a3}.

\quad (a2) Assume $t=4$ and $b=2$. To prove that $h^1(\Ii _{Z_{8,2}}(4)) =0$ (resp. $h^0(\Ii
_{Z_{9,2}}(4))=0$) it is sufficient to specialize
$Z_{8,2}$ to $Z_{4,0}\cup W_1\cup W_2$ (resp. $Z_{4,0}\cup W_1\cup W_2\cup W_3$) with $W_1$ a general union of $3$
schemes of type $(7,2)$ with respect to $H$, $W_2$ a general union of $3$ schemes of type $(4,1)$, $W_3$ a scheme of type $(4,1)$ with respect to $H$ and then apply
first Proposition \ref{ux3} and then Lemma \ref{a3}. Note that $\binom{8}{4} -\deg (Z_{8,2}) =3 \ge 2 = \binom{7}{4} -\deg
(H\cap (W_1\cup W_2))$ and that
$\deg (H\cap (W_1\cup W_2\cup W_3)) =2 =\deg (Z_{9,3})-\binom{8}{4}$; if one of these two numerical checks were wrong the first degeneration would send
$Z_{x,2}$, $x\in \{8,9\}$, in a configuration with both $h^0\ne 0$ and $h^1\ne 0$.

\quad (a3) Assume $t=4$ and $b=1$. To prove that $h^1(\Ii _{Z_{12,2}}(4)) =0$ (i.e. $h^0(\Ii _{Z_{12,2}}(4))=1$ and so $h^0(\Ii
_{Z_{13,2}}(4))=0$) we degenerate $Z_{12,1}$ to $Z_{4,0}\cup M_2\cup M_3$ with $M_3$ a a scheme of type $(2,7)$ with respect
to $H$ and
$M_3$ a general union of $8$ schemes of type $(4,1)$ with respect to $H$. We apply Proposition \ref{ux3} to the intersection
with $H$ and then we apply Lemma \ref{a3}. Note that $\binom{8}{4}-\deg (Z_{12,1}) =1 = \binom{7}{4} -\deg (H\cap (M_2\cup
M_3))$.
To prove $h^1(\Ii _{Z_{4,0}\cup \Res _H(M_2)}(3)) =0$ it is sufficient to note that that (since $Z_{4,0}$ is not related to $H$) $Z_{4,0}\cup \Res _H(M_2)$ is projectively equivalent to a subscheme
of some $Z_{4,1}$.

\quad \emph{Claim 1:} On $\PP^4$ we have $h^1(\Ii _{Z_{4,1}}(3))=0$.

\quad \emph{Proof of Claim 1:} We degenerate $Z_{4,1}$ to $Z_{1,0}\cup M_4\cup M_5$ with $M_4$ a scheme of type $(7,2)$
with respect to $H$ and $M_5$ a general union of $3$ schemes of type $(4,1)$ with respect to $H$. We use Proposition \ref{ux3}
and Lemma \ref{a3}.

\quad (a4) Assume $t=5$. It is sufficient to test all $a, b$ such that $122 \le 5a+9b\le 130$. After each degeneration we use
Proposition \ref{ux3} for $t=5$, Lemma \ref{a3} and the cases $t=3,4$ just done. If
$a
\ge 14$ (and in particular if $b\le 3$) we degenerate
$Z_{a,b}$ to $Z_{a-14,b}\cup M_6$ with $M_6$ a general union of $14$ schemes of type $(4,1)$ with respect to $H$. If $b\ge 8$
we degenerate $Z_{a,b}$ to $Z_{a,b-8}\cup M_7$ with $M_7$ a general union of $8$ schemes of type $(7,2)$ with respect to $H$.
If $4\le b\le 7$ (and hence $a\ge 7$) we degenerate $Z_{a,b}$ to $Z_{a-4,b-7}\cup M_8\cup M_9$ with $M_8$ a general union of $4$
schemes of type $(7,2)$ with respect to $H$ and $M_9$ a general union of $7$ schemes of type $(4,1)$ with respect to $H$.

\quad (a5) Assume $t=6$. It is sufficient to  test all $a, b$ such that $5a+9b \ge 206$. After each degeneration we use
Proposition \ref{ux3} for $t=6$ and Lemma \ref{a3}. If $a\ge 21$  (and in particular if $b\le 3$) we degenerate $Z_{a,b}$ to
$Z_{a-21,b}\cup N_1$ with
$N_1$ a general union of $21$ schemes of type $(4,1)$ with respect to $H$. If $b\ge 12$ we degenerate $Z_{a,b}$ to
$Z_{a,b-12}\cup N_2$ with $N_2$ a general union of $12$ schemes of type $(7,2)$ with respect to $H$. If $4 \le b \le 7$
 we degenerate $Z_{a,b}$ to
$Z_{a-7,b-4}\cup N_3\cup N_4$ with
$N_3$ a general union of $4$ schemes of type $(7,2)$ with respect to $H$ and $N_4$ a general union of $7$ schemes of type
$(4,1)$ with respect to
$H$.

\quad (b) Assume  $n=5$. The case $t=3$ is true by our assumptions on $a, b$, the Alexander-Hirschowitz theorem and
\cite{cgg, bcgi}. We do here the cases
$t=4,5$ (which require the case
$t=4$ in
$H\cong
\PP^4$) because the inductive proof done in steps ({c}) and (d) works verbatim. We have $h^0(\Oo _{\PP^5}(2)) = 21$, $h^0(\Oo _{\PP^5}(3))
= 56$,
$h^0(\Oo _{\PP^5}(4)) = 140$, $h^0(\Oo _{\PP^5}(5)) = 266$. By assumption either $a\ge 10$ or $b\ge 6$.

\quad (b1) Assume $t=4$. By \cite{bcgi} we may assume $a>0$. We may assume $6a+11b \ge 135$. If $a\ge 14$ (and in particular if
$b\le 4$) we degenerate
$Z_{a,b}$ to
$Z_{a-14,b}\cup N_5$ with
$N_5$ a general union of $14$ schemes of type $(5,1)$ with respect to $H$. If $b\ge 5$ and $a\ge 5$ we degenerate $Z_{a,b}$ to
$Z_{a-5,b-5}\cup N_6\cup N_7$ with $N_7$ a general union of $7$ schemes of type $(9,2)$ with respect to $H$ and $N_8$ a
general union of $5$ schemes of type $(5,1)$ with respect to $H$. If $1\le a \le 4$ (and hence $b\ge 8$) we degenerate
$Z_{a,b}$ to
$Z_{a-1,b-8}\cup N_9\cup O_1\cup O_2$ with $N_9$ a general union of $7$ schemes of type $(9,2)$ with respect to $H$, $O_1$ a
scheme of type
$(2,9)$ with respect to $H$ and $O_2$ a scheme of type $(5,1)$ with respect to $H$.

\quad (b2) Assume $t=5$.  We may assume $6a+11b \ge 261$. After each degeneration we use the case $n=4$ and $t=5$ done in
step (a2) and Lemma \ref{a3}.  If $b\ge 14$ we degenerate $Z_{a,b}$
to $Z_{a,b-14}\cup O_3$ with $O_3$ a general union of $14$ schemes of type $(9,2)$ with respect to $H$. If $4\le b\le 13$
(and hence $a\ge 18$)  we degenerate $Z_{a,b}$
to $Z_{a,b-14}\cup O_4\cup O_5$ with $O_5$ a general union of $4$ schemes of type $(9,2)$ with respect to $H$ and $O_5$ a
general union of $13$ schemes of type $(5,1)$ with respect to $H$. If $b\le 3$ (and hence $a\ge 26$) we degenerate $Z_{a,b}$
to to $Z_{a-26,b}\cup O_6\cup O_7$ with $O_6$ a general union of $25$ schemes of type $(5,1)$ with respect to $H$ and $O_7$
a scheme of type $(1,5)$ with respect to $H$.

\quad (c) By steps (a) and (b) we may assume $n\ge 6$ and that the theorem is true for
$\PP^{n-1}$. By Remark ref{stu1} we may reduce the proof of the theorem (preserving the assumption $\max \{(n+1)a, (2n+1)b\}\ge
\binom{n+3}{3}$) to the check of all pairs $(a,b)\in
\NN^2$ such that
$$\binom{n+t}{n} -n \le (n+1)a+(2n+1)b\le \binom{n+t}{n} +n$$
Fix a hyperplane
$H\subset
\PP^n$. The case $t=3$
is true by the Alexander-Hirschowitz theorem (resp.
\cite{av}) if
$a\ge
\lceil
\binom{n+3}{3}/(n+1)\rceil$ (resp. $b\ge
\lceil \binom{n+3}{3}/(2n+1)\rceil$). Thus we may use $t\ge 4$ and that the theorem is true for $H^0(\Oo _{\PP^n}(x))$ for all
integers $x$  such that $3\le x<t$. It is sufficient to handle all $(a,b)$ with $(n+1)a+(2n+1)b \ge \binom{n+t}{n}-n$.

\quad (d) Assume $a\ge  \lceil \binom{n+3}{3}/(n+1)\rceil$. There are unique integers $x, y$ such that $(2n-1)x+2y
=\binom{n+t-1}{n-1}$ and $0\le y \le 2n-2$. Since $\binom{n+t-1}{n-1} \ge 4n-4$, we have $x\ge 0$.

\quad (d1) Assume $b\ge x+y$. We specialize $Z_{a,b}$ to $Z_{a,b-x-y}\cup W\cup W'$, where $W$ is a general union of $x$
schemes of type $(2n-1,2)$ with respect to $H$ and $W'$ is a general union of $y$ schemes of type $(2,2n-1)$ with respect to $H$. By \cite{av} and Lemma
\ref{a2} we have
$h^i(H,\Ii _{W\cap H}(t)) =0$. Thus it is sufficient to prove that $Z_{a,b-x-y}\cup \Res _H(W\cup W')$ gives the expected number
of conditions to $H^0(\Oo _{\PP^n}(t-1))$. By Lemma \ref{a3} it is sufficient to prove that $Z_{a,b-x-y}\cup \Res _H(W')$
gives the expected number of conditions to $H^0(\Oo _{\PP^n}(t-1))$ and that $$h^0(\Ii _{Z_{a,b-x-y}}(t-2)) \le \max \{0,
\binom{n+t-1}{n} - y(2n-1) - (n+1)a-(2n+1)(b-x-y)\}.$$ The latter condition is satisfied by the inductive assumption if $t\ge 5$,
while if $t=4$ we use that $a>n$ and hence $h^0(\Ii _{Z_{a,b-x-y}}(t-2))=0$.

\quad \emph{Claim 1:} $Z_{a,b-x-y}\cup \Res _H(W')$ gives the expected number
of conditions to $H^0(\Oo _{\PP^n}(t-1))$.

\quad \emph{Proof of Claim 1:} Set $B:= \Res _H(W')$. Since $y\le 2n-2$, we have $2x\ge 4n-4$. Since $(n+1)a +(2n+1)b \le
\binom{n+t}{n}+2n-2$ and
$\deg (Z_{a,b-x-y}\cup
\Res _H(W')) = (n+1)a +(2n+1)b -\binom{n+t-1}{n-1} -2x$, we need to prove that $h^1(\Ii_{Z_{a,b-x-y}\cup B}(t-1)) =0$.
Write $\binom{n+t-2}{n-1} + (2n-1)y + (2n-1)w+ z$ with $w, z$ integers  and $0\le z\le 2n-2$. Since $y\le 2n-2$, we have
$w\ge 0$. 

\quad (d1.1) Assume $b-x-y\ge w$. We specialize $Z_{a,b-x-y}$ to $Z_{a,b-x-y-w}\cup W''$ with $W''$ general
union of $w$ schemes of type $(2n-1,2)$ with respect to $H$. Since $\deg (W''\cap H) \le \binom{n+t-2}{n-1}$, the inductive
assumption gives $h^1(H,\Ii _{H\cap (Z_{a,b-x-y-w}\cup W')}(t-1)) =0$ and hence it is sufficient to prove $h^1(\Ii
_{Z_{a,b-x-y-w}}(t-2))=0$.
This is not numerically impossible, because $\binom{n+t-2}{n-1}-\deg (W''\cap H) =u\le 2n-2$ and $2x\ge 2n-2 + \deg (Z_{a,b})
-\binom{n+t}{n} +2n-2$. If $t\ge 5$ we may use the inductive assumption. If $t=4$ we have $h^0(\Ii _{Z_{a,0}}(2)) =0$ and to get
$h^0(\Ii _{Z_{a,b}}(4))=0$
we insert two other schemes of type $(n,1)$ to handle $B$.  

\quad (d1.2) Now assume $b<x+y+w$. Write $\binom{n+t-2}{t-1} = (2n-1)(b-x)+ nu_1+u_2$ with $u_1, u_2$ integers and $0\le u_2
\le n-1$. 

Assume for the moment $a\ge u_1$. We specialize $Z_{a,b-x-y}\cup B$ to a general
$E:= Z_{a-u_1,0}\cup B\cup M_0 \cup M'$ with $M_0$ a general union of $u_1$ schemes of type $(n,1)$ with respect to $H$ and
$M'$ a general union of $u_2$ schemes of type $(1,n)$ with respect to $H$. We want to prove that $h^1(H,\Ii _{H\cap E}(t)) =0$ and $h^1(\Ii _{\Res _H(E)}(t-1)) =0$ (these vanishings
would prove Claim 1 in this case). To apply the inductive assumption to $E\cap H$ we need that either $(2n+1)(b-x)\ge
\binom{n+2}{3}$ or $(n+1)u_1 \ge \binom{n+2}{3}$. Assume that both inequalities fails and so
$(n+1)u_1 + (2n-1)(b-x) \le 2\binom{n+3}{3}-2$. Since $u_2\le n-1$, we have
$(2n-1)(b-x) +nu _1\ge \binom{n+t-1}{n-1}-n+1$. The right hand side of the last inequality is an increasing function of $t$
and so, taking $t=5$, we get
$2\binom{n+3}{3} -2 \ge \binom{n+4}{n-1}-n+1$, which is false for $n\ge 5$.
The scheme
$\Res _H(E)$ is a general union of $Z_{a-u_1,0}$, $u_1$ points of $H$ and $b-x$ tangent vectors of $H$. We use the
Alexander-Hirschowitz theorem,  Lemma \ref{a3} and the obvious inequality $u_1+2(b-x) \le \binom{n+t-1}{n}
-\binom{n+t-2}{n}$, except that if $t=4$ we also need $h^0(\Ii _{Z_{a-u_1,0}}(2)) \le \max \{0,\binom{n+3}{3} -(n+1)(a-u_1)
-2(b-x)\}$. If $a-u_1\ge n+1$, then $h^0(\Ii _{Z_{a-u_1,0}}(2))=0$. If $a-u_1 \le n$, then $h^0(\Ii
_{Z_{a-u_1,0}}(2))=\binom{n+2-a+u_1}{2}$.

Now assume $u_1>a$. We specialize $B\cup Z_{a,x-y-z}$ to $E':= B\cup M'\cup M''$ with $M'$ $b-x-y$ schemes of type $(2n-1,2)$
with respect to $H$ and $M''$ a general union of $a$ schemes of type $(n,1)$. We conclude as above (to handle $H\cap E'$ we
use that $E'\cap H\subset E\cap H$ and we proved that $h^1(H,\Ii _{E\cap H}(t)) =0$), concluding the proof of Claim 1.

\quad (d2) Assume $x \le b<x+y$. Write $2(y-b+x) = nw_1+w_2$ with $w_1\in \{0,1\}$ and $0\le w_2\le n-1$. We specialize
$Z_{a,b}$ to
$J:= Z_{a-x-y+b,0}\cup M_1\cup M_2\cup M_3\cup M_4$ with
$M_1$ a general union of
$x$ schemes of type $(2n-1,2)$, $M_2$ a general union of $x+y-b$ schemes of type $(2,2n-1)$, $M_3$ a general union of
$w_1$ schemes of type $(n,1)$ and $M_4$ a general union of $w_2$ schemes of type $(1,n)$. Note that $\deg (J\cap
H)=\binom{n+t-1}{n-1}$. By Lemma \ref{a3} to prove that $Z_{a,b}$ gives independent conditions to $H^0(\Oo _{\PP^n}(t))$
it is sufficient to prove that $h^1(H,\Ii _{H\cap (M_1\cup M_3)}(t)) =0$ (true by the inductive assumption, because $e_1\le 1$
and so $|a-x-y+b|\ge \binom{n+2}{3}$ since $t\ge 5$)
and that
$h^1(\Ii _{Z_{a-x-y+b,0}}(t-1)) =0$. The last vanishing is not numerically impossible, because $\deg (Z_{a-x-y+b,0})=\deg
(Z_{a,b}) -\binom{n+t-1}{n-1} -e_2-2x$ with $x\ge n$ and $\deg (Z_{a,b}) \le \binom{n+t}{n}+2n-2$. Thus $h^1(\Ii
_{Z_{a-x-y+b,0}}(t-1)) =0$ by the Alexander-Hirschowitz theorem.

\quad (d3) Assume $b<x$. Set $u:= \lfloor (\binom{n+t-1}{n-1}-(2n-1)b)/2\rfloor$ and $v:=
\binom{n+t-1}{n-1}-(2n-1)b -2u$. We specialize $Z_{a,b}$ to $Z_{a-u-v,0}\cup M_5\cup M_6$ with $M_5$ general union of $u$
schemes of type $(n,1)$ with respect to $H$ and $M_6$ general union of $v$ schemes of type $(1,n)$ with respect to $H$.
We have $h^i(H,\Ii _{H\cap M_5\cup M_6}(t)) =0$, $i=0,1$. For the details of the conclusion, see the proof of Claim 1 in step
(d1).
\end{proof}

\section{Proof of Theorem \ref{i1}}

\begin{proof}[Proof of Theorem \ref{i1}:]
Let $x$ be a positive integer such that $L^{\otimes x}$ is very ample. Applying the theorem (assumed to be true) to the line
bundle
$L':= L^{\otimes x}$ and finitely many times to the same zero-dimensional scheme $W$ and to finitely many line bundles
$M\otimes L^{\otimes i}$,
$0\le i\le x-1$, we reduce to the case in which $L$ is very ample. This reduction step is exactly the reduction step in
\cite[\S 7]{ah2} and we believe that this should be better known. Set
$n:=
\dim X$. As in \cite{ah2} we use induction on $n$. We use the same strategy: we try to land after a controlled number of steps in a simpler situation, which may be considered as a winning situation, e.g. because there are no tangential schemes.

\quad (a) Assume $n=1$. In this case $Z_{a,b}$ is a general union of connected schemes of multiplicity $2$ and $3$ whose
support is general in $X_{\mathrm{reg}}$. In characteristic $0$ it is sufficient to quote \cite{cm}. 

\quad (b) Assume $n>1$ and that  Theorem \ref{i1} is true for all projective varieties of dimension $<n$. Since $L$ is very ample, there is $D\in |L|$ such that $D\cap W=\emptyset$. Set $M_1:= M_{|D}$ and $R:= L_{|D}$. Fix an integer
$k_1>0$ such that $h^1(\Ii _W\otimes M\otimes L^{\otimes t})=0$ for all $t\ge k_1-1$ and (inductive assumption) for all
$(c,d)\in \NN^2$ and all integers $t\ge k_1$ either $h^1(D,\Ii _{A_{c,d}}\otimes M_1\otimes R^{\otimes t}) =0$ or $h^0(D,\Ii
_{A_{c,d}}\otimes M_1\otimes R^{\otimes t}) =0$, where $A_{c,d}\subset D$ is a general union of $c$ 2-points of $D$ and $d$
tangential schemes of $D$. Set $\beta := h^0(\Ii _W\otimes M\otimes L^{\otimes k_1})$. For all integers $t\ge k_1$ set $A(t):=
h^0(D,M_1\otimes L^{\otimes t})$ and
$B(t):= h^0(\Ii _W\otimes M\otimes L^{\otimes t})$. We have $B(t) -B(t-1)=A(t)$ for all $t\ge k_1$, because $h^1(\Ii _W\otimes
M\otimes L^{\otimes t})=0$ for all
$t\ge k_1-1$. Since the function $A(t)$ is strictly increasing, there is an integer $k_2>k_1$ such that $A(k_2) \ge
(2n+1)\beta +2n$. By \cite{ah2} (case with multiplicity $2$) taking a high tensor power of $L$ instead of $L$ we may also
assume that the thesis holds with
$k_2=1$ for all
$(a,b)$ with $b=0$. Since $\chi (R^{\otimes t}) =A(t)$ for all $t\ge k_1$, Riemann-Roch gives that $A(t)$ increases like
$ct^{n-1}/(n-1)!$ and $B(t)$ increases like $ct^n/n!$, where $c = L\cdots L$ is the self-intersection number of $n$ copies of
$L$. Thus we may find integers $k_0>k_3\ge k'>0$ such that for all $t\ge k_3$ we have
\begin{enumerate}
\item $A(t) \le (2n-1)A(t-2)$;
\item $A(t) \ge n(2n+1)+4n$;
\item $B(k_0) \ge 5nB(k_3)$.
\end{enumerate}

Fix an integer $t\ge k_0$ and $(a,b)\in \NN^2$. If $\Ll$ is any line bundle on $X$ and $Z, Z'$ are
zero-dimensional subchemes of $X$ such that $Z\subseteq Z'$ and $h^1(\Ii _{Z'}\otimes \Ll)=0$ (resp. $h^0(\Ii _Z\otimes
\Ll)=0$), then
$h^1(\Ii _Z\otimes \Ll)=0$ (resp. $h^0(\Ii _{Z'}\otimes \Ll)=0$). Thus increasing or decreasing both elements of $(a,b)$ we
see that it is sufficient to prove the statement for all $(a,b)$ such that $B(t)-n \le (n+1)a +(2n+1)b \le B(t)+n$ (Remark
\ref{stu1}). We only use the last inequality, but a weaker inequality like  $(n+1)a +(2n+1)b \le B(t)+10^{10}n$ would work. 

\quad (b1) We define integers $b_x\ge 0$ and $e_x\ge 0$ for all $x\ge 1$ until we will arrive at some integer $y>0$
such that $b< b_1+e_1+\cdots + b_y+e_y$ or at step $t-k_3$, i.e. the steps in which we would define $b_{t-k_3}$ and $e_{t-k_3}$. At that point we go to step (b2). Since $2n-1$ is odd,
there are integers
$b_1$ and
$e_1$ such that $A(t) =(2n-1)b_1 + 2e_1$ and $0 \le e_1 \le 2n-2$. Since $A(t)\ge n(2n-1)+2(2n-2)$, we have $b_1\ge n$. We
specialize
$Z_{a,b}$ to the union $W_1$ of $Z_{a,b-b_1-e_1}$, the union $E_1$ of $b_1$ schemes of type $(2n-1,2)$ with respect to $D$ and
the union $F_1$ of $e_1$ schemes of type $(2,2n-1)$ with respect to $D$. Since $W\cap D=\emptyset$, $\deg (W_1\cap D)=A(t)$ and
$W_1\cap D$ is a general union of
$b_1$ tangential schemes of $D$ and $e_1$ tangent vectors of $D$, Lemma \ref{a2} gives $h^i(D,\Ii _{W_1\cap D}\otimes
M_1\otimes R^{\otimes t})=0$. Thus it is sufficient to prove that either $h^0(\Ii _{W\cup \Res _D(W_1)}\otimes M\otimes
L^{\otimes (t-1)}) =0$ or $h^1(\Ii _{W\cup \Res _D(W_1)}\otimes M\otimes
L^{\otimes (t-1)}) =0$. Set $W_2:= Z_{a,b-b_1-e_1}\cup \Res _D(F_1)$. Since $\Res _D(E_1)$ is a general union of $b_1$ tangent
vectors of $D$, by Lemma \ref{a3} it is sufficient to prove that $h^1(\Ii _{W\cup W_2}\otimes L^{\otimes (t-1)})=0$
and that $h^0(\Ii _{W\cup W_2}\otimes L^{\otimes (t-1)})\le \max \{0,B(t-1) - \deg (\Res _D(W_1))\}$. We call $\S (t-2)$ the
latter condition and we will check it in step (b2). Now we continue the proof that $h^1(\Ii _{W\cup W_2}\otimes L^{\otimes
(t-1)})=0$. Let $b_2$ and $e_2$ be the only integers such that $A(t-1) = (2n-1)e_1 + (2n-1)b_2 +2e_2$ with $0\le e_2\le 2n-2$.
Since $t-1 \ge k_3$ and $e_1\le 2n-2$, we have $b_2\ge 0$. We specialize $W_2$ to a general union of $Z_{a,b-b_1-b_2-e_1-e_2}$,
a general union of $b_2$ schemes of type $(2n-1,2)$ with respect to $D$ and $e_2$ schemes of type $(2,2n-1)$ with respect to
$D$.
Note that $\deg (D\cap (E_2\cup F_2)) =A(t-1)$. Note that $\Res _D(\Res _D(F_1)) =\emptyset$. By Lemma \ref{a2} it is
sufficient to prove that
$h^1(\Ii _{W\cup Z_{a,b-b_1-b_2-a_1-a_2}\cup \Res _D(E_2)\cup \Res _D(F_2)}\otimes M\otimes L^{\otimes (t-2)})=0$. By Lemma
\ref{a3} it is sufficient to prove that  $h^1(\Ii _{W\cup Z_{a,b-b_1-b_2-a_1-a_2}\cup \Res _D(F_2)}\otimes
M\otimes L^{\otimes (t-2)})=0$ and that $h^0(\Ii _{W\cup Z_{a,b-b_1-b_2-a_1-a_2}\cup \Res _D(F_2)}\otimes
M\otimes L^{\otimes (t-3)})\le \max \{0,B(t-3)-\deg (W\cup Z_{a,b-b_1-b_2-a_1-a_2})\}$; we call $\S (t-3)$ the last condition
and we will check it in step (b2). Here we continue the proof of the $h^1$-vanishing. Let $b_3$ and $e_3$ be the only integer
such that $(2n-1)e_2+(2n-1)b_3+2e_3 =A(t-2)$. Since $t-2 \ge k_0$, we have $b_3\ge 0$. We specialize $W_3$
to a general union of of $Z_{a,b-b_1-b_2-b_3-e_1-e_2-e_3}$, the union $E_3$ of $b_3$ schemes of type $(2n-1,2)$ with respect
to
$D$ and the union $F_3$ of $e_3$ schemes of type $(2,2n-1)$ with respect to $D$. Then we continue. We need to check that
this construction stops before using $L^{\otimes k_3-1}$. 

\quad \emph{Claim 1:} There is an integer $y\le t-k_3-1$ such that $b< b_1+e_1+\cdots +b_y+e_y$.

\quad \emph{Proof of Claim 1:} We have $(n+1)a + (2n+1)b \le B(t)+2n-2$. Since $(2n-1)(b_i+e_i) \ge A(t-i)$, for all $i \le t-k_1$ we have  $b_1+e_1+\cdots +b_y+e_y \ge (\sum _{i=1}^{y} A(t+1-1))/(2n-1) =(B(t)-B(t-y))/(2n-1)$. Since $(n+1)a +(2n+1)b \le
B(t)+n$, $t\ge k_0$ and $B(k_0)\ge 5nB(k_3)$, we get Claim 1.

\quad (b1.1) By Claim 1 there is a maximal integer $z$ such that $b\ge b_1+e_1+\cdots +b_z+e_z$ and (modulo the check of $\S
(x)$ for
$t-z-2 \le x \le t-2$) to conclude for $Z_{a,b}$ it would be sufficient to prove that $h^1(\Ii _{W\cup Z_{a,b-b_1-e_1-\cdots -b_z-e_z}\cup B}\otimes M\otimes L^{\otimes (t-z)})=0$, 
where $B$ is a general union of $e_z$ tangential schemes of $D$. Set $w:= b-b_1-e_1-\cdots -b_z-e_z$. 

\quad (b1.1.1.1) First assume $(2n+1)w\le
A(t-1-z)-\deg (B)$. Write $A(t-1-z) =(2n+1)w +\deg (B) +na_1+c_1$ with $(a_1,c_1)\in \NN^2$ and $0\le c_1\le n-1$. Assume for
the moment $a\ge a_1+c_1$. We specialize
$Z_{a,w}\cup B$ to $E:= B\cup Z_{a-a_1-a_2,0}\cup G\cup G'$, where $G$ is a general union of $a_1$ schemes of type $(n,1)$ with respect to $D$ and $G'$ is a general union of  of $c_1$ schemes of type $(1,n)$ with respect to $D$. By the inductive assumption we have $h^i(D,\Ii _{E\cap D}\otimes M_1\otimes
R^{\otimes t-z}) =0$, $i=0,1$, and so (modulo the checks of $\S x$ for $t-z-1\le x\le t$) we conclude if $h^1(\Ii _{\Res
_D(E)}(t-z-1)) =0$. This is true, because we only have $2$-points, we use \cite{ah2} and from $L^{\otimes (t-1)}$ down we always
had schemes $E_i$ for $L^{\otimes i}$ with $\deg (E_i) \le B(i)$.

Now assume $a<a_1+c_1$. We specialize $Z_{a,w}\cup B$ to $E':= B\cup G_1\cup G_2$, where $G_1$ is a general union of $\min \{a_1,a\}$ schemes of type $(n,1)$, $G_2=\emptyset$ if $a\le a_1$ and $G_2$ is a general union of $a-a_1$ schemes of type $(1,n)$. We get $h^1(D,\Ii _{E\cap D}\otimes M_1\otimes
R^{\otimes t-z}) =0$ and obviously $h^1(\Ii _{W\cup E'}\otimes M\otimes L^{\otimes t-z-1})=0$.

\quad (b1.1.1.2) Now assume $(2n+1)w \ge
A(t-1-z)-\deg (B)$. Assume for the moment $a\ge 2e_{z+1}$. We specialize $Z_{a,w}\cup B$ to $E'':= Z_{a-2e_z,w-x_{z
+1}}\cup G_3\cup G_4$ with $G_3$ a general union of $b_{z+1}$  schemes of type $(2n-1,2)$ with respect to $D$ and $G_4$ a general union of
$2e_{z+1}$ schemes of type $(1,n)$ with respect to $D$. By the inductive assumption on $D$ we have $h^i(D,\Ii _{E''}\otimes
M_1\otimes R^{\otimes (t-z)})=0$, $i=1,2$, while for $\Res _D(E'')$ we always have at most $2n-3$ tangential schemes. We
degenerate $\Res _D(E'')$ and after taking $\Res _D$ of the degenerated subscheme we go to a case $Z_{a',0}$.

If $a< 2e_z$ we specialize $Z_{a,w}\cup B$ to $F:= Z_{a-2e_z,w-x_{z+1}}\cup G_3\cup G_4$ with $G_3$ a general union of $b_{z+1}$  schemes of type $(2n-1,2)$ with respect to $D$ and $G_4$ a general union of
$2e_{z+1}$ schemes of type $(1,n)$ with respect to $D$. We only have $h^1(D,\Ii _{D\cap F}\otimes M_1\otimes R^{\otimes
(t-z-1)}) =0$ (the $h^0$ is often $\ne 0$, say $h^0=\epsilon$), but the next steps are even easier because at the next step we may lose $2n-2$ at each step and hence we continue
without using differential Horace (no scheme of type $(2,2n-1)$ or $(1,n)$) until we arrive at a scheme $Z_{0,0}=\emptyset$ before using $L^{\otimes (k_3-1)}$.

\quad (b2) We need to check the uses of Lemma \ref{a3} and that the construction stops in a winning position before we use
$L^{\otimes k_3-1}$. We stress again that all steps not using Lemma
\ref{a3} either are in a case in which the degree of the scheme, $G$, involved for $L^{\otimes x}$ is such that $\deg (G\cap
D) = A(x)$ (and we have $h^i(D,\Ii _{F\cap D}\otimes M_1\otimes R^{\otimes x})=0$, $i=0,1$) and $\deg (G) - B(x) =\deg (\Res
_D(G))-B(x-1)$ or we had a configuration, $A$, with not enough $2$-points or tangential schemes to get such an equality and
we specialize to a configuration $A'$ with only $h^1(D,\Ii _{A'}\otimes M_1\otimes R^{\otimes x})=0$, but $\Res _D(A')$ is so
small that obviously $h^1(\Ii _{W\cup \Res _D(A')}\otimes M\otimes L^{\otimes (x-1)}) =0$. Consider any use of Lemma \ref{a3}
 that requires to check the inequality 
$h^0(\Ii _{W\cup F_x}\otimes M\otimes L^{\otimes x}) \le \max \{0, B(x+1) -\deg (W\cup F_x) -\deg (G_x)\}$ with $G_x \subset D$,
and $G_x$ a general union of points and tangent vectors of
$D$. At the beginning we observed that after the very first step we are in the range in which the degrees of the
zero-dimensional schemes, $\eta$, are at most the $h^0$ of the line bundle $\Ll$ we are using and so we need to prove
$h^1(\Ii _\eta \otimes \Ll))=0$, except for the use of Lemma \ref{a3} for which we need a weaker statement: the condition
$h^0\le \max \{0,\ast \}$ is equivalent to $h^1 \le \max \{0,\ast \ast \}$ for a suitable integer $\ast \ast$. First consider the
first
$z+1$ steps. Here at least
$b_1+\cdots +b_z+e_1+\cdots +b_z \ge (B(t) - B(t-z-1))/(2n-1)$ tangential schemes are fully used, i.e. we either used their intersection
with
$D$ and then also their residual with respect to $D$, or first we used the tangential scheme of
$D$ (its degree is $2n-1$) and then the residual (a tangent vector of $D$) is omitted quoting Lemma \ref{a3}. In the first
$z$ steps (the ones involving
$(b_i,e_i)$) we omit (quoting Lemma \ref{a3}) $b_1+\cdots +b_z$ tangent vectors. Suppose that when using $L^{\otimes (x+1)}$ we
have a zero-dimensional scheme
$F_{x+1}\cup G_{x+1}\cup A_{x+1}$  with $F_{x+1}\cap D =\emptyset$, $\deg (G_{x+1})$ very small, say $\deg (G_{x+1}) \le
(2n-2)(2n+1)$ and $A_{x+1}\subset D$ with $A_{x+1}$ general union of points and tangent vectors of $D$. Lemma \ref{a3} say that
to prove
$h^1(\Ii _{W\cup F_{x+1}\cup G_{x+1}\cup A_{x+1}}\otimes M\otimes L^{\otimes (x+1)})=0$ it is sufficient to prove
that $h^1(\Ii _{W\cup F_{x+1}\cup G_{x+1}}\otimes M\otimes L^{\otimes (x+1)})=0$ and $h^0(\Ii _{W\cup F_{x+1}\cup \Res
_D(G_{x+1})}\otimes M\otimes L^{\otimes x}) \le \max \{0,B(x+1) -\deg (W\cup F_{x+1}\cup G_{x+1}\cup
A_{x+1})\}$. Since $\deg (G_{x+1}\cup A_{x+1}) \ll A(t+1) = B(t+1)-B(t)$, the proof of the upper bound for $h^0$ (i.e. the
upper bound for $h^1$) is translated in a similar question, but for $L^{\otimes x}$ instead of $L^{\otimes (x+1)}$. We use the
same method for $L^{\otimes x}$ for $F_{x+1}\cup \Res
_D(G_{x+1})$. If the proof that $h^1(\Ii _{W\cup F_{x+1}\cup G_{x+1}}\otimes M\otimes L^{\otimes (x+1)})=0$ stops after at
most $\tau$ steps to $W\cup \emptyset$ or to $W\cup Z_{a',0}$ (i.e. for $L^{\otimes (x+1-\tau)}$), the proof for the upper
bound of $h^1$ stops after at most $\tau$ steps, i.e. we use $L^{\otimes (t+1-\tau)}$ but not $L^{\otimes (x-\tau)}$. Call $U(i)$ the scheme such that we need to
prove $h^1(\Ii _{W\cup U(i)}\otimes M\otimes L^{\otimes i})=0$ and set $\tau (i):= B(i) -\deg (W) -\deg (U(i))$. We started
with $U(t):= Z_{a,b}$ and we have $\tau (t) \ge 0$. Going down in the first $z+1$ steps we delete (by Lemma \ref{a3}) at
least $b_1+\dots +b_z$ tangent vectors and hence after at most $z+1$ steps, we get $\tau (z+1) \ge 2(b_1+\cdots +b_z)$.
When  for $L^{\otimes i}$ we use $d$ schemes of type $(n,1)$ with respect to $D$, then they give a contribution of $d$ to the
integer $\tau (i+2)-\tau (i)$. To conclude for $L^{\otimes i}$ it is sufficient to have $U(i)=\emptyset$ and so it is
sufficient to have $\tau (i) \ge B(i)-\deg(W)$. Since $(2n+1)/2 \le n+1$, if for some $L^{\otimes i}$, $i \le t-z-2$, we have not yet
won, we have $\tau (i) \ge B(t)/2n$ and so $\tau (i) \ge B(i)$. We arrive at such an integer $i$ before $k_3-1$, because $t\ge
k_0$ and $B(k_0)\ge 5nB(k_3)$.
\end{proof}


\begin{thebibliography}{99}

\bibitem{ab} H. Abo and M. C. Brambilla,  On the dimensions of secant varieties of Segre-Veronese varieties, Ann. Mat. Pura
Appl. (4) 192 (2013), no. 1, 61--92.

\bibitem{av} H. Abo and N. Vannieuwenhoven, Most secant varieties of tangential varieties to Veronese varieties are non-defective, Trans. Amer. Math. Soc. 370 (2018), no. 1, 393--420.

\bibitem{abre} S. Abrescia, About defectivity of certain Segre-Veronese varieties, Canad. J.Math. 60 (2008), no. 5, 961--974.

\bibitem{a}  B. \r{A}dlandsvik, Joins and higher secant varieties, Math.
Scand. 61 (1987), 213--222.


\bibitem{ah} J. Alexander and A. Hirschowitz, 
Polynomial interpolation in several variables, 
J. Algebraic Geom. 4 (1995), no. 2, 201--222.

\bibitem{ah2} J. Alexander and A. Hirschowitz, An asymptotic vanishing theorem for generic unions of multiple points, Invent.
Math. 140 (2000), 303--325.

\bibitem{bb} E. Ballico and M. C. Brambilla, Postulation of general quartuple fat points schemes in $\PP^3$, J. Pure Appl. Algebra  213 (2009), 1002-1012.

\bibitem{bbcs} E. Ballico, M. C. Brambilla, A. Caruso and M. Sala, Postulation of general quintuple fat points schemes in $\PP^3$, J. Algebra 366 (2014), no. 2, 857--874.

\bibitem{bf} E. Ballico and C. Fontanari, A Terracini lemma for osculating spaces with applications to Veronese surfaces, J. Pure Appl. Algebra 195 (2006), 1--6.

\bibitem{bcgi} A. Bernardi, M. V. Catalisano, A. Gimigliano and M. Id\`{a}, Secant varieties to osculating varieties of Veronese embeddings, J. Algebra 321 (2009), no. 3, 982--1004.

\bibitem{bo} M. C. Brambilla and G. Ottaviani, 
On the   Alexander-Hirschowitz Theorem, J. Pure Appl. Algebra 212 
(2008), no. 5,  1229--1251.




\bibitem{cgg} M. V. Catalisano, A. V. Geramita and A. Gimigliano, On the secant varieties to the tangential varieties of a Veronesean, Proc. Amer. Math. Soc. 130 (2002), no. 4, 975--985.
\bibitem{co} M. V. Catalisano and A. Oneto, Tangential varieties of Segre-Veronese surfaces are never defective, Rev. Mat.
Complut. 33 (2020), no. 1, 295--324.

\bibitem{cm} C. Ciliberto and R. Miranda, Interpolations on curvilinear schemes, J. Algebra 203 (1998), no. 2, 677--678.

\bibitem{e} L. Evain, Computing limit linear series with infinitesimal methods, Ann. Inst. Fourier 57 (2007), no. 6,
1947--1974.

\bibitem{h} R. Hartshorne,  Algebraic Geometry, Springer, Berlin, 1977.



\bibitem{ik} A. Iarrobino and V. Kanev, Power sums, Gorenstein algebras, and determinantal loci,
Lecture Notes in
Mathematics, vol. 1721, Springer-Verlag, Berlin, 1999, Appendix C by 
Iarrobino and Steven L. Kleiman.

\bibitem{laf} A. Laface, On linear systems of curves on rational scrolls, 
Geom. Dedicata 90 (2002), 127--144.

\bibitem{lp} A. Laface and E. Postinghel, Secant varieties of Segre-Veronese embeddings of $(\PP^1)^r$, Math. Ann. 256 (2013), 1455--1470.


\bibitem{l} J.M. Landsberg, 
Tensors: Geometry and Applications, Graduate Studies in Mathematics 128, 
Amer. Math. Soc. Providence, 2012.

\bibitem{r1} J. Ro\'{e}, Limit linear systems and applications, arXiv:0602213.

\bibitem{r2} J. Ro\'{e}, Maximal ranks for schemes of small multiplicities by \'{E}vain's differential Horace method, Trans. Amer. Math. Soc. 366 (2014), no. 2, 857--874.
\end{thebibliography}
\end{document}